\documentclass[final]{siamart1116}

\usepackage{lipsum}
\usepackage{amsfonts}
\usepackage{graphicx}
\usepackage{epstopdf}
\usepackage{algorithmic}
\ifpdf
  \DeclareGraphicsExtensions{.eps,.pdf,.png,.jpg}
\else
  \DeclareGraphicsExtensions{.eps}
\fi

\numberwithin{theorem}{section}

\newcommand{\TheTitle}{Solving Nonlinear Optimal Control Problems With State and Control Delays by Shooting Methods Combined with Numerical Continuation on the Delays}
\newcommand{\TheShortTitle}{Optimal Control Problems With State and Control Delays}
\newcommand{\TheAuthors}{R. Bonalli, B. H\'{e}riss\'{e}, and E. Tr\'{e}lat}

\headers{\TheShortTitle}{\TheAuthors}

\title{{\TheTitle}
}

\author{
  Riccardo Bonalli\thanks{Sorbonne Universit\'{e}s, UPMC Univ Paris 06, CNRS UMR 7598, Laboratoire Jacques-Louis Lions, F-75005, Paris, France and Onera - The French Aerospace Lab, F - 91761 Palaiseau, France (\email{riccardo.bonalli@onera.fr})}
  \and
  Bruno H\'{e}riss\'{e}\thanks{Onera - The French Aerospace Lab, F - 91761 Palaiseau, France (\email{bruno.herisse@onera.fr}).}
  \and
  Emmanuel Tr\'{e}lat\thanks{Sorbonne Universit\'{e}s, UPMC Univ Paris 06, CNRS UMR 7598, Laboratoire Jacques-Louis Lions, F-75005, Paris, France (\email{emmanuel.trelat@upmc.fr}, \url{https://www.ljll.math.upmc.fr/trelat/}).}
}

\usepackage{amsopn}

\ifpdf
\hypersetup{
  pdftitle={\TheTitle},
  pdfauthor={\TheAuthors}
}
\fi

\usepackage{dsfont}

\usepackage{mathtools}

\begin{document}

\maketitle

\begin{abstract}
In this paper we introduce a new procedure to solve nonlinear optimal control problems with delays which exploits indirect methods combined with numerical homotopy procedures. It is known that solving this kind of problems via indirect methods (which arise from the Pontryagin Maximum Principle) is complex and computationally demanding because their implementation is faced to two main difficulties: the extremal equations involve forward and backward terms, and besides, the related shooting method has to be carefully initialized. Here, starting from the solution of the non-delayed version of the optimal control problem, delays are introduced by a numerical continuation. This creates a sequence of optimal delayed solutions that converges to the desired solution. We establish a convergence theorem ensuring the continuous dependence w.r.t. the delay of the optimal state, of the optimal control (in a weak sense) and of the corresponding adjoint vector. The convergence of the adjoint vector represents the most challenging step to prove and it is crucial for the well-posedness of the proposed homotopy procedure. Two numerical examples are proposed and analyzed to show the efficiency of this approach.
\end{abstract}



\section{Introduction} \label{secIntro}

\subsection{Delayed Optimal Control Problems}

Let $n$, $m$ be positive integers, $\Delta$ a positive real number, $\Omega \subseteq \mathbb{R}^m$ a measurable subset and define an initial state function $\phi^1(\cdot) \in C^0([-\Delta,0],\mathbb{R}^n)$ and an initial control function $\phi^2(\cdot) \in L^{\infty}([-\Delta,0],\Omega)$. For every $\tau = (\tau^1,\tau^2) \in [0,\Delta]^2$ and every positive final time $T$, consider the following nonlinear control system on $\mathbb{R}^n$ with constant delays
\begin{equation} \label{dynDelay}
\begin{cases}
\dot{x}_{\tau}(t) = f(t,x_{\tau}(t),x_{\tau}(t-\tau^1),u_{\tau}(t),u_{\tau}(t-\tau^2)) \ , \quad t \in [0,T] \medskip \\
x_{\tau}(t) = \phi^1(t) \ , \ t \in [-\Delta,0] \quad , \quad  u_{\tau}(t) = \phi^2(t) \ , \ t \in [-\Delta,0) \medskip \\
u_{\tau}(\cdot) \in L^{\infty}([-\Delta,T],\Omega)
\end{cases}
\end{equation}
where $f(t,x,y,u,v) = f_1(t,x,y,u) + f_2(t,x,y,v)$ and $f_1 : \mathbb{R} \times \mathbb{R}^{2n} \times \mathbb{R}^{m} \rightarrow \mathbb{R}^{n}$, $f_2 : \mathbb{R} \times \mathbb{R}^{2n} \times \mathbb{R}^{m} \rightarrow \mathbb{R}^{n}$ are of class (at least) $C^2$ w.r.t. their second and third variables. Control systems (\ref{dynDelay}) play an important role describing many phenomena in physics, biology and economics (see, e.g. \cite{Malek-Zavarei:1987:TSA:576371}).

Let $M$ be a subset of $\mathbb{R}^{n}$. Assume that $M$ is reachable from $\phi^1(\cdot)$ for the control system (\ref{dynDelay}), that is, for every $\tau = (\tau^1,\tau^2) \in [0,\Delta]^2$, there exists a final time $T_{\tau}$ and a control $u_{\tau}(\cdot) \in L^{\infty}([-\Delta,T_{\tau}],\Omega)$, such that the trajectory $x_{\tau}(\cdot)$, solution of (\ref{dynDelay}) on $[0,T_{\tau}]$, satisfies $x_{\tau}(T_{\tau}) \in M$. Such a control is called \textit{admissible} and we denote by $\mathcal{U}^{\tau}_{T_{\tau},\mathbb{R}^m}$ the set of all admissible controls of (\ref{dynDelay}) defined on $[-\Delta,T_{\tau}]$ taking their values in $\mathbb{R}^m$ while $\mathcal{U}^{\tau}_{T_{\tau},\Omega}$ denotes the set of all admissible controls of (\ref{dynDelay}) defined on $[-\Delta,T_{\tau}]$ taking their values in $\Omega$. Then, $\mathcal{U}^{\tau}_{\mathbb{R}^m} := \cup_{T > 0} \mathcal{U}^{\tau}_{T,\mathbb{R}^m}$ and $\mathcal{U}^{\tau}_{\Omega} := \cup_{T > 0} \mathcal{U}^{\tau}_{T,\Omega}$.

Given a couple $\tau = (\tau^1,\tau^2) \in [0,\Delta]^2$ of delays, we consider the Optimal Control Problem with Delays (\textbf{OCP})$_{\tau}$ consisting in steering the control system (\ref{dynDelay}) to $M$, while minimizing the cost function
\begin{equation} \label{costDelay}
C_{T_{\tau}}(\tau,u_{\tau}(\cdot)) := \int_0^{T_{\tau}} f^0(t,x_{\tau}(t),x_{\tau}(t-\tau^1),u_{\tau}(t),u_{\tau}(t-\tau^2)) \; dt
\end{equation}
where $f^0(t,x,y,u,v) = f^0_1(t,x,y,u) + f^0_2(t,x,y,v)$ and $f^0_1 : \mathbb{R} \times \mathbb{R}^{2n} \times \mathbb{R}^{m} \rightarrow \mathbb{R}$, $f^0_2 : \mathbb{R} \times \mathbb{R}^{2n} \times \mathbb{R}^{m} \rightarrow \mathbb{R}$ are of class (at least) $C^2$ w.r.t. their second and third variables. The final time $T_{\tau}$ may be fixed or not.

The literature is abundant of numerical methods to solve (\textbf{OCP})$_{\tau}$. Most of them rely on \textit{direct methods}, which basically consist in discretizing all the variables concerned and in reducing (\textbf{OCP})$_{\tau}$ to a finite dimensional problem. The works \cite{wong1985optimal, lee1993numerical, chen2000numerical, lee2006semi, gollmann2009optimal} develop several numerical techniques to convert the optimal control problem with delays into nonlinear constrained optimization problems. On the other hand, \cite{horng1985analysis, hwang1985optimal, perng1986direct, khellat2009optimal, haddadi2012optimal} propose different approaches that approximate the solution of (\textbf{OCP})$_{\tau}$ by truncated orthogonal series and reduce the optimal control problem with delays to a system of algebraic equations. Yet, other contributions (see, e.g. \cite{banks1979approximation}) propose an approximating sequence of non-delayed optimal control problems whose solutions converge to the optimal solution of (\textbf{OCP})$_{\tau}$. However, the dimension induced by these approaches becomes as larger as the discretization is finer.

Some applications, many of which are found in aerospace engineering, like atmospheric reentry and satellite launching (for example in \cite{trelat2012optimal}), require great accuracy which can be reached by \textit{indirect methods}. Moreover, the computational load needed to compute indirect methods remains minimal if compared to the one used to obtain a good solution with direct approaches. It is then interesting to solve efficiently (\textbf{OCP})$_{\tau}$ via these procedures.

\subsection{Indirect Methods Applied to (\textbf{OCP})$_{\tau}$}

The core of indirect methods relies on solving, thanks to Newton-like algorithms, the two-point or multi-point boundary value problem which arises from necessary optimality conditions coming from the application of the \textit{Pontryagin Maximum Principle} (PMP) \cite{pontryagin1987mathematical}.

The paper \cite{kharatishvili1961maximum} was first to provide a maximum principle for optimal control problems with a constant state delay while \cite{guinn1976reduction} obtains the same conditions by a simple substitution-like method. In \cite{kharatishvili1967maximum} a similar result is achieved for control problems with pure control delays. In \cite{halanay1968optimal, soliman1972optimal}, necessary conditions are obtained for optimal control problems with multiple constant delays in state and control variables. Moreover, \cite{banks1968necessary} derives a maximum principle for control systems with a time-dependent delay in the state variable. Finally, \cite{gollmann2014theory} provides necessary conditions for optimal control problems with multiple constant delays and mixed control-state constraints.

The advantages of indirect methods, whose more basic version is known as \textit{shooting method}, are their extremely good numerical accuracy and the fact that, if they converge, the convergence is very quick. Indeed, since they rely on the Newton method, they inherit the convergence properties of the Newton method. Nevertheless, their main drawback is related to their difficult \textit{initialization} (see for example \cite{trelat2012optimal}). This is pointed out as soon as the necessary optimality conditions are computed on (\textbf{OCP})$_{\tau}$.

It is known that (see, e.g. \cite{pontryagin1987mathematical,boccia2013free,gollmann2014theory}), if $(x_{\tau}(\cdot),u_{\tau}(\cdot))$, $\tau = (\tau^1,\tau^2) \in [0,\Delta]^2$ is an optimal solution of (\textbf{OCP})$_{\tau}$ with optimal final time $T_{\tau}$, there exist a nonpositive scalar $p^0_{\tau}$ and an absolutely continuous mapping $p_{\tau} : [0,T_{\tau}] \rightarrow \mathbb{R}^n$ called \textit{adjoint vector}, with $(p_{\tau}(\cdot),p^0_{\tau}) \neq (0,0)$, such that the so-called \textit{extremal} $(x_{\tau}(\cdot),p_{\tau}(\cdot),p^0_{\tau},u_{\tau}(\cdot))$ satisfies
\begin{eqnarray} \label{dynDual}
\hspace{15pt}
\begin{cases}
\displaystyle \dot{x}_{\tau}(t) = \displaystyle \frac{\partial H}{\partial p}(t,x_{\tau}(t),x_{\tau}(t-\tau^1),p_{\tau}(t),p^0_{\tau},u_{\tau}(t),u_{\tau}(t-\tau^2)) \ , \ t \in [0,T_{\tau}] \medskip \\
\displaystyle \dot{p}_{\tau}(t) = \displaystyle -\frac{\partial H}{\partial x}(t,x_{\tau}(t),x_{\tau}(t-\tau^1),p_{\tau}(t),p^0_{\tau},u_{\tau}(t),u_{\tau}(t-\tau^2)) \medskip \\
\displaystyle \hspace{14pt} -\frac{\partial H}{\partial y}(t+\tau^1,x_{\tau}(t+\tau^1),x_{\tau}(t),p_{\tau}(t+\tau^1),p^0_{\tau},u_{\tau}(t+\tau^1),u_{\tau}(t+\tau^1-\tau^2)) \ , \medskip \\
\hspace{14pt} t \in [0,T_{\tau}-\tau^1] \medskip \\
\displaystyle \dot{p}_{\tau}(t) = \displaystyle -\frac{\partial H}{\partial x}(t,x_{\tau}(t),x_{\tau}(t-\tau^1),p_{\tau}(t),p^0_{\tau},u_{\tau}(t),u_{\tau}(t-\tau^2)) \ , \medskip \\
\hspace{14pt} t \in (T_{\tau}-\tau^1, T_{\tau}]
\end{cases}
\end{eqnarray}
where $H(t,x,y,p,p^0,u,v) = \langle p , f(t,x,y,u,v) \rangle + p^0 f^0(t,x,y,u,v)$ is the Hamiltonian, and the maximization condition
\begin{eqnarray} \label{maxCond}
\footnotesize
\displaystyle \hspace{12pt} &H&(t,x_{\tau}(t),x_{\tau}(t-\tau^1),p_{\tau}(t),p^0_{\tau},u_{\tau}(t),u_{\tau}(t-\tau^2)) \medskip \\
\displaystyle &+& \mathds{1}_{[0,T_{\tau}-\tau^2]} H(t+\tau^2,x_{\tau}(t+\tau^2),x_{\tau}(t+\tau^2-\tau^1),p_{\tau}(t+\tau^2),p^0_{\tau},u_{\tau}(t+\tau^2),u_{\tau}(t)) \nonumber \medskip \\
\displaystyle &\geq& H(t,x_{\tau}(t),x_{\tau}(t-\tau^1),p_{\tau}(t),p^0_{\tau},v,u_{\tau}(t-\tau^2)) \nonumber \medskip \\
\displaystyle &+& \mathds{1}_{[0,T_{\tau}-\tau^2]} H(t+\tau^2,x_{\tau}(t+\tau^2),x_{\tau}(t+\tau^2-\tau^1),p_{\tau}(t+\tau^2),p^0_{\tau},u_{\tau}(t+\tau^2),v) \nonumber \medskip \\
&\forall& v \in \Omega \nonumber
\end{eqnarray}
holds almost everywhere on $[0,T_{\tau}]$. Moreover, if the final time is free and, without loss of generality, one supposes that $T_{\tau}$ and $T_{\tau} - \tau$ are points of continuity of $u_{\tau}(\cdot)$,
\begin{equation} \label{freeT}
H(T_{\tau},x_{\tau}(T_{\tau}),x_{\tau}(T_{\tau}-\tau^1),p_{\tau}(T_{\tau}),p^0_{\tau},u_{\tau}(T_{\tau}),u_{\tau}(T_{\tau}-\tau^2)) = 0
\end{equation}
(see \cite{boccia2013free} for a more general condition using the concept of Lebesgue approximate continuity). The extremal $(x_{\tau}(\cdot),p_{\tau}(\cdot),p^0_{\tau},u_{\tau}(\cdot))$ is said \textit{normal} whenever $p^0_{\tau} \neq 0$, and in that case it is usual to normalize the adjoint vector so that $p^0_{\tau} = -1$; otherwise it is said \textit{abnormal}.

Assuming that $u_{\tau}(\cdot)$ is known as a function of $x_{\tau}(\cdot)$ and $p_{\tau}(\cdot)$ (by the maximization condition (\ref{maxCond})), each iteration of a shooting method consists in solving the coupled dynamics (\ref{dynDual}), where a value of $p_{\tau}(T_{\tau})$ is provided. This means that one has to solve a \textit{Differential-Difference Boundary Value Problem} (DDBVP) where both forward and backward terms of time appear within \textit{mixed type differential equations}. The difficulty is then the lack of global information which forbids a purely local integration by usual iterative methods for ODEs. Some techniques to solve mixed type differential equations were developed. For example, \cite{mallet2003mixed} proposes an analytical decomposition of the solutions as sums of \textit{forward solutions} and \textit{backward solutions}, while \cite{ford2010numerical} provides a solving numerical scheme. However, these approaches treat either only linear cases or the inversion of matrices whose dimension increases as much as the numerical accuracy raises.

In order to initialize correctly a shooting method for (\ref{dynDual}), a guess of the final value of the adjoint vector $p_{\tau}(T_{\tau})$ is not sufficient, but rather, a good numerical guess of the whole function $p_{\tau}(\cdot)$ must be provided to make the procedure converge. This represents an additional difficulty with respect to the usual shooting method and it requires a global discretization of (\ref{dynDual}).

It seems that this topic has been little addressed in the literature. The paper \cite{bader1985solving} proposes a collocation methods to solve the DDBVP arising from (\ref{dynDual}) that turns out to be successful to solve several optimal control problems with delays. However, as a consequence of the collocation method, the degree of interpolating polynomials grows up fast for hard problems. Moreover, a precomputation of points where the solution of (\ref{dynDual}) has discontinuous derivative is needed to make the whole approach feasible, intensifying the quantity of numerical computations.

\subsection{Numerical Homotopy Approach}

The basic idea of homotopy methods is to solve a difficult problem step by step, starting from a simpler problem, by parameter deformation. Theory and practice of continuation methods are well known (see, e.g. \cite{allgower2003introduction}). Combined with the shooting problem derived from the PMP, a homotopy method consists in deforming the problem into a simpler one (that can be easily solved) and then in solving a series of shooting problems step by step to come back to the original problem. The main difficulty of homotopy methods lies in the choice of a sufficiently regular deformation that allows the convergence of the homotopy method. The starting problem should be easy to solve, and the path between this starting problem and the original problem should be handy to model. This path is parametrized by a parameter denoted $\lambda$ and, when the homotopic parameter $\lambda$ is a real number and the path is linear in $\lambda$, the homotopy method is rather called a \textit{continuation method}.

Consider the Optimal Control Problem \textit{without} Delays (\textbf{OCP})$\equiv$(\textbf{OCP})$_0$ which consists of steering to $M$ the control system
\begin{equation} \label{dynZero}
\begin{cases}
\dot{x}(t) = f(t,x(t),x(t),u(t),u(t)) \ , \quad t \in [0,T] \medskip \\
x(t) = \phi^1(0) \ , \quad u(\cdot) \in L^{\infty}([0,T],\Omega)
\end{cases}
\end{equation}
while minimizing the cost function
\begin{equation} \label{costZero}
C_{T}(u(\cdot)) := C_{T}(0,u(\cdot)) = \int_0^{T} f^0(t,x(t),x(t),u(t),u(t)) \; dt \quad .
\end{equation}

In many situations, exploiting the non-delayed version of the PMP mixed to other techniques (such as geometric control, dynamical system theory applied to mission design, etc., we refer the reader to \cite{trelat2012optimal} for a survey on these procedures), one is able to initialize efficiently a shooting method on (\textbf{OCP}). Thus, it is legitimate to wonder if one may solve (\textbf{OCP})$_{\tau}$ by indirect methods starting a homotopy method where $\tau$ represents the deformation parameter and (\textbf{OCP}) is taken as the starting problem. This approach is a way to address the flaw of indirect methods applied to (\textbf{OCP})$_{\tau}$: on one hand, the global adjoint vector of (\textbf{OCP}) could be used to inizialize efficiently a shooting method on (\ref{dynDual}) (for $\| \tau \|$ small enough) and, on the other hand, we could solve (\ref{dynDual}) via usual iterative methods for ODEs (for example, by using the global state solution at the previous iteration).

However, one should be careful when using homotopy methods. As we pointed out previously, the existence of a sufficiently regular deformation curve of delays $\tau$ that allows the convergence of the method must be ensured. In \cite{trelat2012optimal}, it was proved that, in the case of unconstrained optimal control problems without delays where $M = \{ x_1 \}$, the existence of a parameter deformation curve is equivalent to ask that neither abnormal minimizers nor conjugate points occur along the homotopy path. Some similar assumptions must be made to apply this procedure to solve succesfully (\textbf{OCP})$_{\tau}$ by indirect methods.

\subsection{Main Contribution and Paper Structure}

The idea proposed in this paper consists in introducing a general method that allows to solve successfully (\textbf{OCP})$_{\tau}$ using indirect methods combined with homotopy procedures, with $\tau$ as deformation parameter, starting from the solution of its non-delayed version (\textbf{OCP}).

The main contribution of the paper is a convergence theorem that ensures the continuous dependence w.r.t. the delay of the optimal state, the optimal control (in a weak sense) and of the corresponding adjoint vector of (\textbf{OCP})$_{\tau}$. This ensures to reach the optimal solution of (\textbf{OCP})$_{\tau}$ starting from the optimal solution of (\textbf{OCP}) iteratively by travelling across a sequence $(\tau_k)_{\mathbb{N}}$ converging to $\tau$. The most challenging and most important nontrivial conclusion is the continuous dependence of the adjoint vectors of (\textbf{OCP})$_{\tau}$ w.r.t. $\tau$. This last fact is crucial because it allows indirect methods to solve (\textbf{OCP})$_{\tau}$ starting a homotopy method on $\tau$ with (\textbf{OCP}) as initial problem.

The article is structured as follows. Section 2 presents the assumptions and the statement of the convergence theorem; moreover, a practical algorithm to solve (\textbf{OCP})$_{\tau}$ by homotopy is provided. In Section 3 the efficiency of this approach is illustrated by testing the proposed algorithm on two examples. Finally, in the appendices, the technical details of the proof of the main result are provided.

\section{Convergence Theorem for (\textbf{OCP})$_{\tau}$} \label{sectionResult}

Within the proposed convergence result, it is crucial to split the case in which the delay $\tau^2$ on the control variable appears from the one which considers only pure state delays. The context of control delays reveals to be more complex, especially, in proving the existence of optimal control for (\textbf{OCP})$_{\tau}$. Indeed, a standard approach to prove existence would consider usual Filippov's assumptions (as in the classical reference \cite{filippov1962certain}) which, in the case of control delays, must be extended. In particular, using the \textit{Guinn's reduction} (see, e.g. \cite{guinn1976reduction}), the control system with delays results to be equivalent to a non-delayed system with a larger number of variables depending on the value of $\tau^2$. Such extension was used in \cite{nababan1979filippov}. However, it is not difficult to see that the usual assumption concerning the convexity of the epigraph of the extended dynamics is not sufficient to prove Lemma 2.1 in \cite{nababan1979filippov}. More details are provided in Remark \ref{remarkGuinn}, in Section \ref{existenceSect}.

\subsection{Main Result} \label{subTheo}

We make the following assumptions.

\textbf{Common assumptions:}
\begin{enumerate}
\item[$(A_1)$] $\Omega$ is a compact and convex subset of $\mathbb{R}^m$ and $M$ is a compact subset of $\mathbb{R}^n$;
\item[$(A_2)$] \begingroup \small(\textbf{OCP})\endgroup has a unique solution \begingroup \small$(x(\cdot),u(\cdot))$\endgroup defined on a neighborhood of \begingroup \small$[-\Delta,T]$\endgroup;
\item[$(A_3)$] The optimal trajectory $x(\cdot)$ has a unique extremal lift (up to a multiplicative scalar) defined on $[0,T]$, which is normal, denoted $(x(\cdot),p(\cdot),-1,u(\cdot))$, solution of the Maximum Principle;
\item[$(A_4)$] There exists a positive real number $b$ such that, for every $\tau = (\tau^1,\tau^2) \in [0,\Delta]^2$ and every $v(\cdot) \in \mathcal{U}^\tau_{\Omega}$, denoting $x_{\tau,v}(\cdot)$ the related trajectory arising from the control system (\ref{dynDelay}) with final time $T_{\tau,v}$, we have
$$
\forall t \in [-\Delta,T_{\tau,v}] : \quad  T + T_{\tau,v} + \| x_{\tau,v}(t) \| \leq b \quad .
$$
\end{enumerate}

\textbf{In case of pure state delays:}
\begin{enumerate}
\item[$(B_1)$] For every delay $\tau$, every optimal control $u_{\tau}(\cdot)$ of (\textbf{OCP})$_{\tau}$ is continuous;
\item[$(B_2)$] The sets
\begingroup
\small
$$
\bigg\{ \left(f_1(t,x,y,u),f^0_1(t,x,y,u)+\gamma,\frac{\partial \tilde f_1}{\partial x}(t,x,y,u),\frac{\partial \tilde f_1}{\partial y}(t,x,y,u)\right) : u \in \Omega \ , \ \gamma \geq 0 \bigg\}
$$
\endgroup
are convex for every $t \in \mathbb{R}$ and every $(x,y) \in \mathbb{R}^{2n}$, where $\tilde f_1(t,x,y,u) = (f_1(t,x,y,u),f^0_1(t,x,y,u))$.
\end{enumerate}

\textbf{In case of delays both in state and control variables:}
\begin{enumerate}
\item[$(C_1)$] For every delay $\tau$, every optimal control $u_{\tau}(\cdot)$ of (\textbf{OCP})$_{\tau}$ takes its values at extremal points of $\Omega$. Moreover, the optimal final time $T_{\tau}$ and $T_{\tau} - \tau$ are points of continuity of $u_{\tau}(\cdot)$;
\item[$(C_2)$] The vector field $f$ and the cost function $f^0$ are \textit{locally Lipschitz} w.r.t. $(u,v)$ i.e., for every $(t,x,y,u_0,v_0) \in \mathbb{R}^5$ there exist a neighborhood $W$ of $(u_0,v_0)$ and a continuous function $\alpha(t,x,y)$, such that
$$
\| f(t,x,y,u_1,v_1) - f(t,x,y,u_2,v_2) \| \le \alpha(t,x,y) \Big( \| u_1 - u_2 \| + \| v_1 - v_2 \| \Big)
$$
for every $(u_1,v_1) \; , \; (u_2,v_2) \in W$ (the same statement holds for $f^0$);
\item[$(C_3)$] The sets
$$
\Big\{ (f(t,x,y,u,v),f^0(t,x,y,u,v)+\gamma) : u , v \in \Omega \ , \ \gamma \geq 0 \Big\}
$$
$$
\Big\{ (f(t,x,y,u,u),f^0(t,x,y,u,u)+\gamma) : u \in \Omega \ , \ \gamma \geq 0 \Big\}
$$
are convex for every $t \in \mathbb{R}$ and every $(x,y) \in \mathbb{R}^{2n}$.
\end{enumerate}

Some remarks on these assumptions are in order.

First of all, assumptions $(A_2)$ and $(A_3)$ on the uniqueness of the solution of (\textbf{OCP}) and on the uniqueness of its extremal lift are related to the differentiability properties of the value function (see, e.g. \cite{clarke1987relationship,aubin2009set,emmRifford}). They are standard in optimization and are just made to keep a nice statement (see Theorem \ref{theoMain}). These assumptions can be weakened as follows. If we replace $(A_2)$ and $(A_3)$ with the assumption "every extremal lift of every solution of (\textbf{OCP}) is normal", then the conclusion provided in Theorem \ref{theoMain} hereafter still holds, except that the convergence properties must be written in terms of closure points. The proof of this fact follows the same guideline used to prove Theorem \ref{theoMain} and we avoid to report the details.

Assumptions $(B_1)$ and $(C_1)$ play a complementary role in proving the convergence property for the adjoint vectors. Moreover, Assumption $(C_1)$ becomes also crucial to ensure the convergence of optimal controls and trajectories when considering delays both in state and control variables. Without this assumption of nonsingular controls, proving these last convergences becomes a hard task. The issue is related to the following fact. Let $X$, $Y$ be Banach spaces and $F : X \rightarrow Y$ be a continuous map. Suppose that $(x_k)_{k \in \mathbb{N}} \subseteq X$ is a sequence such that $x_k \rightharpoonup x$ and $F(x_k) \rightharpoonup F(\bar x)$ for some $x, \bar x \in X$. Then, in general, we cannot ensure that $x \equiv \bar x$. A way to overcome this flaw is to ensure the equivalence between weak convergence and strong convergence under some additional assumptions, and, in our main result, this is achieved thanks to $(C_1)$ (see, e.g. \cite{visintin1984strong}).

\begin{theorem} \label{theoMain}
Assume that assumptions $(A_1)$, $(A_2)$, $(A_3)$ and $(A_4)$ hold.

Consider first the context of pure state delays i.e., problems (\textbf{OCP})$_{\tau}$ such that $\tau = (\tau^1,0)$ and $f_2(t,x,y,u) = f^0_2(t,x,y,u) = 0$, and assume that assumptions $(B_1)$ and $(B_2)$ hold. Then, there exists $\tau_0 > 0$ such that, for every $\tau = (\tau^1,0) \in (0,\tau_0) \times \{ 0 \}$, (\textbf{OCP})$_{\tau}$ has at least one solution $(x_{\tau}(\cdot),u_{\tau}(\cdot))$ whose arc is defined on $[-\Delta,T_{\tau}]$, every extremal lift of which is normal. Let $(x_{\tau}(\cdot),p_{\tau}(\cdot),-1,u_{\tau}(\cdot))$ be such a normal extremal lift. Then, up to continuous extensions on $[-\Delta,T]$, as $\tau$ tends to 0,
\begin{itemize}
\item $T_{\tau}$ converges to $T$;
\item $x_{\tau}(\cdot)$ converges uniformly to $x(\cdot)$;
\item $p_{\tau}(\cdot)$ converges uniformly to $p(\cdot)$;
\item $\dot{x}_{\tau}(\cdot)$ converges to $\dot{x}(\cdot)$ in $L^{\infty}$ for the weak star topology.
\end{itemize}
If the final time $T$ of (\textbf{OCP}) is fixed, then $T_{\tau} = T$ for every $\tau \in (0,\tau_0)^2$.

On the other hand, consider general problems (\textbf{OCP})$_{\tau}$ with delays $\tau = (\tau^1,\tau^2)$ in both state and control variables. If one assumes that, for every $\tau \in [0,\Delta]^2$, if (\textbf{OCP})$_{\tau}$ is controllable then it admits an optimal solution, then, under assumptions $(C_1)$, $(C_2)$ and $(C_3)$, there exists $\tau_0 > 0$ such that, for every $\tau = (\tau^1,\tau^2) \in (0,\tau_0)^2$, the same conclusions on the convergences given above hold and, in addition, as $\tau$ tends to 0, $u_{\tau}(\cdot)$ converges to $u(\cdot)$ almost everywhere in $[0,T]$. Moreover, if dynamics $f$ and cost $f^0$ are affine w.r.t. in the two control variables, the existence of an optimal solution $(x_{\tau}(\cdot),u_{\tau}(\cdot))$ for every $\tau = (\tau^1,\tau^2) \in (0,\tau_0)^2$ is ensured.

Finally, for every $\bar \tau \in [-\Delta,0]$, by extending to the delay $\bar \tau$ all the previous assumptions, we have that the optimal solutions $(x_{\tau}(\cdot),u_{\tau}(\cdot))$ of (\textbf{OCP})$_{\tau}$ (or $(x_{\tau}(\cdot),\dot{x}_{\tau}(\cdot))$ in the case of pure state delays) and their related adjoint vectors $p_{\tau}(\cdot)$ are continuous w.r.t. $\tau$ at $\bar \tau$ for the above topologies.
\end{theorem}

The proof of Theorem \ref{theoMain} is technical and lenghty. We report it in Appendix \ref{appProof}.

The last statement of Theorem \ref{theoMain} (the continuous dependence w.r.t. $\tau$, for every $\tau \in [-\Delta,0]$) is the most general conclusion achieved and extends the first part of the theorem. The proof of this generalization follows the same guidelines of the proof of the continuity at $\tau = 0$ and we avoid to report the details.

We want to stress the fact that the continuous dependence w.r.t. $\tau$ of the adjoint vectors of (\textbf{OCP})$_{\tau}$ represents the most challenging and the most important result achieved by Theorem \ref{theoMain}. It represents the essential step that allows the proposed homotopy method to converge robustly for every, small enough, couple of delays $\tau$. The proof of this fact is not easy. An accurate analysis of the convergence of Pontryagin cones in the case of the delayed version of the PMP is required.

\begin{remark} \label{remarkExample}
Theorem \ref{theoMain} can be extended to obtain stronger convergence conclusions, by using weaker assumptions, in the particular case of dynamics $f$ that are affine in the two control variables, and costs of type
$$
\int_0^{T_{\tau}} \Big[ C_1 \| x_{\tau}(t) \|^2 + C_2 \| x_{\tau}(t-\tau^1) \|^2 + C_3 \| u_{\tau}(t) \|^2 + C_4 \| u_{\tau}(t-\tau^2) \|^2 \Big] \; dt
$$
where $C_1, C_2, C_3, C_4 \ge 0$ are constants. Indeed, considering assumptions $(A_1)$-$(A_4)$ and either $(B_1)$ or $(C_1)$, the convergence properties established in Theorem \ref{theoMain} for $x_{\tau}(\cdot)$ and $p_{\tau}(\cdot)$ still hold and, moreover, $u_{\tau}(\cdot)$ converges to $u(\cdot)$ in $L^2$ for the weak topology, as $\tau$ tends to 0. The proof of this fact arises easily adapting the scheme in Appendix \ref{appProof}. For sake of brevity, we do not give these technical details.
\end{remark}

\subsection{The Related Algorithm and Its Convergence} \label{subSectConvAlgo}

Exploiting the statement of Theorem \ref{theoMain}, we may conceive a general algorithm, based on indirect methods, capable of solving (\textbf{OCP})$_{\tau}$ by applying homotopy procedures on parameter $\tau$, starting from the solution of its non-delayed version (\textbf{OCP}).

As we explained in the previous sections, the critical behavior coming out from this approach consists of the integration of mixed-type equations that arise from System (\ref{dynDual}). The previous convergence result suggests us the idea that we may solve (\ref{dynDual}) via usual iterative methods for ODEs, for example, by using the global state solution at the previous iteration. Moreover, the global adjoint vector of (\textbf{OCP}) could be used to inizialize, from the beginning, the whole shooting.

\begingroup
\begin{algorithm}
\caption{- Indirect Numerical Homotopy Algorithm for solving (\textbf{OCP})$_{\tau}$}
\label{alg}
\begin{algorithmic}
\vspace{5pt}
\STATE{Set $k=0$, $\tau_k = 0$, and fix an integer $k_{\max}$. Solve (\textbf{OCP}) by indirect methods and denote $(x_{\tau_k}(\cdot),p_{\tau_k}(\cdot),p^0_{\tau_k},u_{\tau_k}(\cdot))$ its numerical extremal solution.}
\vspace{5pt}
\WHILE{$\| \tau_k \| < \| \tau \|$ and $k < k_{\max}$}
\STATE{\vspace{5pt}A.1) Compute $\tau_{k+1}$ by predictor-corrector methods and steplength adaptation;}

\STATE{\vspace{5pt}A.2) Solve recursively (\textbf{OCP})$_{\tau_{k+1}}$ by indirect methods initialized by $p_{\tau_{k}}(\cdot)$ i.e., fixing $k$, an integer $i_{\max}$, a tolerance $\varepsilon$ and setting $i=1$, $p^i_{\tau_{k+1}}(\cdot) := p_{\tau_{k}}(\cdot)$, do \vspace{5pt}

\WHILE{$\|p^{i+1}_{\tau_{k+1}}(\cdot) - p^i_{\tau_{k+1}}(\cdot)\|_{C^0} \geq \varepsilon$, $x^{i+1}_{\tau_{k+1}}(T^{i+1}_{\tau_{k+1}}) \notin M_1$ and $i < i_{\max}$}
\STATE{\vspace{5pt}B.1) Guess $\bar p^{i+1}_{\tau_{k+1}}(\cdot)$ from $p^i_{\tau_{k+1}}(\cdot)$ (via Newton-like algorithms);}

\STATE{\vspace{5pt}B.2) Using the PMP, express $u^{i+1}_{\tau_{k+1}}(\cdot)$ as function of $(x^{i+1}_{\tau_{k+1}}(\cdot),\bar p^{i+1}_{\tau_{k+1}}(\cdot),p^{0,{i+1}}_{\tau_{k+1}})$;}
\STATE{\vspace{5pt}B.3) Solve
\begin{eqnarray*}
	\begin{cases}
		\displaystyle \dot{x}^{i+1}_{\tau_{k+1}}(t) = \displaystyle \frac{\partial H}{\partial p}(\cdot,x^{i+1}_{\tau_{k+1}}(\cdot),\bar p^{i+1}_{\tau_{k+1}}(\cdot),p^{0,{i+1}}_{\tau_{k+1}},u^{i+1}_{\tau_{k+1}}(\cdot))(t) \medskip \\
		x^{i+1}_{\tau_{k+1}}(t) = \phi^1(t) \ , \ t \in [-\Delta,0]
	\end{cases}
\end{eqnarray*}
Then, with $x^{i+1}_{\tau_{k+1}}(\cdot)$ as solution of the previous system, solve
\begin{eqnarray*}
	\begin{cases}
		\displaystyle \dot{p}^{i+1}_{\tau_{k+1}}(t) = \displaystyle -\frac{\partial H}{\partial x}(\cdot,x^{i+1}_{\tau_{k+1}}(\cdot),p^{i+1}_{\tau_{k+1}}(\cdot),p^{0,{i+1}}_{\tau_{k+1}},u^{i+1}_{\tau_{k+1}}(\cdot))(t) \medskip \\
		\displaystyle \hspace{38pt} -\frac{\partial H}{\partial y}(\cdot,x^{i+1}_{\tau_{k+1}}(\cdot),p^{i+1}_{\tau_{k+1}}(\cdot),p^{i+1}_{\tau_{k+1}}(\cdot-\tau^1_{k+1}),p^{0,{i+1}}_{\tau_{k+1}},u^{i+1}_{\tau_{k+1}}(\cdot))(t+\tau^1_{k+1}) \, , \medskip \\
		\hspace{38pt} t \in [0,T^{i+1}_{\tau_{k+1}}-\tau^1_{k+1}] \medskip \\
		\displaystyle \dot{p}^{i+1}_{\tau_{k+1}}(t) = \displaystyle -\frac{\partial H}{\partial x}(\cdot,x^{i+1}_{\tau_{k+1}}(\cdot),p^{i+1}_{\tau_{k+1}}(\cdot),p^{0,{i+1}}_{\tau_{k+1}},u^{i+1}_{\tau_{k+1}}(\cdot))(t) \, , \medskip \\
		\hspace{38pt} t \in (T^{i+1}_{\tau_{k+1}}-\tau^1_{k+1}, T^{i+1}_{\tau_{k+1}}] \medskip \\
		p^{i+1}_{\tau_{k+1}}(T^{i+1}_{\tau_{k+1}}) = \bar p^{i+1}_{\tau_{k+1}}(T^{i+1}_{\tau_{k+1}})
	\end{cases}
\end{eqnarray*}}
\STATE{B.4) $p^i_{\tau_{k+1}}(\cdot) = p^{i+1}_{\tau_{k+1}}(\cdot)$ and $i \rightarrow i+1$.\vspace{5pt}}
\ENDWHILE}
\STATE{\vspace{5pt}A.3) $p_{\tau_k}(\cdot) = p_{\tau_{k+1}}(\cdot)$ and $k \rightarrow k+1$.\vspace{5pt}}
\ENDWHILE
\vspace{5pt}
\end{algorithmic}
\end{algorithm}
\endgroup

These considerations lead us to Algorithm \ref{alg}. To prove the convergence of Algorithm \ref{alg} we apply Theorem \ref{theoMain}. We focus on the case of general state and control delays, highlighting that the same conclusion holds for problems with pure state delays provided that optimal controls can be expressed as continuous functions of the state and the adjoint vector (by using the maximality condition on the Hamiltonian).

Suppose that assumptions $(A_1)$, $(A_2)$, $(A_3)$, $(A_4)$, $(C_1)$, $(C_2)$ and $(C_3)$ hold and that the delay $\tau = (\tau^1,\tau^2) \in [0,\Delta]^2$ considered is such that $\tau \in (0,\tau_0)^2$. Then, we know that for every $\varepsilon = (\varepsilon^1,\varepsilon^2)$ in the open ball $B_{\| \tau \|}(0,0)$, (\textbf{OCP})$_{\varepsilon}$ has at least an optimal solution with normal extremal lift. The first consequence is that, referring to Algorithm \ref{alg}, we can put $p^{0,{i}}_{\tau_{k}} = -1$ for every integers $k$, $i$. Thanks to Theorem \ref{theoMain}, $x_{\varepsilon}(\cdot) \xrightarrow{C^0} x(\cdot)$, $u_{\varepsilon}(\cdot) \xrightarrow{a.e.} u(\cdot)$ and $p_{\varepsilon}(\cdot) \xrightarrow{C^0} p(\cdot)$ as soon as $\varepsilon \rightarrow (0,0)$. Then, the indirect method inside Algorithm \ref{alg} results to be well defined and well initialized by the adjoint vector $p(\cdot)$ of (\textbf{OCP}). Indeed, necessarily, the algorithm will travel backward one of the subsequence converging to the solution of (\textbf{OCP}). Since for every sequence $(\varepsilon_k)_{k \in \mathbb{N}}$ converging to $(0,0)$ the related extremal lift $(x_{\varepsilon}(\cdot),p_{\varepsilon}(\cdot),-1,u_{\varepsilon}(\cdot))$ of (\textbf{OCP})$_{\varepsilon}$ converge to the one of (\textbf{OCP}) (for the evident topologies), every homotopy methods on $\tau$ lead to the same optimal solution of (\textbf{OCP})$_{\tau}$.

\begin{remark}
It is interesting to remark that, at least formally, there are no difficulties to apply Algorithm \ref{alg} to more general (\textbf{OCP})$_{\tau}$ which consider locally bounded varying delays that are functions of the time and the state i.e., $\tau : \mathbb{R} \times \mathbb{R}^n \rightarrow [-\Delta,0]^2 : (t,x) \mapsto \tau(t,x)$. In this context, some relations close to (\ref{dynDual})-(\ref{freeT}) are still provided (see, e.g. \cite{asher1971optimal}), so that, the proposed numerical continuation scheme remains well-defined.
\end{remark}

\section{Numerical Example}

In order to prove effectiveness and robustness of our approach, we test it on two examples. As a matter of standard analysis for numerical approaches to solve optimal control problems with delays, we follow the guideline provided by \cite{gollmann2009optimal}. The first test is an academic example while the second one considers the nontrivial problem consisting of a continuous nonlinear two-stage stirred tank reactor system (CSTR), proposed by \cite{dadebo1992optimal} and \cite{oh1976optimal}.

We stress the fact that, in this paper, we are interested in solving an optimal control problem with delays (\textbf{OCP})$_{\tau}$ starting from its non-delayed version (\textbf{OCP}), without taking care of how (\textbf{OCP}) is solved. Even if we are aware of the fact that this task is far from being easy, here, we focus our attention on the performance achieved once the solution of (\textbf{OCP}) is known. However, as suggested in Section \ref{secIntro}, in many situations one is able to initialize correctly a shooting method on (\textbf{OCP}) (see \cite{trelat2012optimal}).

\subsection{Setting Preliminaries}

The numerical examples proposed are solved applying verbatim Algorithm \ref{alg}. Good solutions are obtained using a basic linear continuation on $\tau$. Moreover, an explicit second-order Runge-Kutta method is handled to solve all the ODEs coming from the dual formulation while the routine \textit{hybrd} \cite{hybrd} is used to solve the shooting problem. The procedure is initialized using the solution of (\textbf{OCP}) provided by the optimization software AMPL \cite{AMPL} combined with the interior point solver IPOPT \cite{wachter2006implementation}.

We stress the fact that one has to be careful when passing the numerical approximation of the extremals in step B.3) of the previous algorithm. Indeed, it is known that, using collocation methods like Runge-Kutta schemes, the error between the solution and its numerical approximation remains bounded throughout $[0,T]$ and decreases with $h^p$, where $h$ is the time step while $p$ is the order of the method, only if this numerical approximation is obtained by interpolating the numerical values within each subinterval of integration with a polynomial of order $p$. From this remark, it is straightforward that the dimension of the shooting considered in Algorithm \ref{alg}, not only it increases w.r.t. $1/h$, but it is also proportional to $p$. In the particular case of an explicit second-order Runge-Kutta method, the dimension of shootings is bounded above by $2 n/h$ (where $n$ is the dimension of the state).

The numerical calculations are computed on a machine Intel(R) Xeon(R) CPU E5-1607 v2 \@ 3.00GHz, with 8.00 Gb of RAM.

\subsection{Analytical Example}

Consider the Optimal Control Problem with Delays (\textbf{OCP})$^1_{\tau}$ which consists in minimizing the cost $J^1(x_{\tau},u_{\tau}) =\int_{0}^{3} (x^2_{\tau}(t) + u^2_{\tau}(t)) \; dt$ subject to
\begin{equation*}
\begin{cases}
\dot{x}_{\tau}(t) = x_{\tau}(t - \tau^1) u_{\tau}(t - \tau^2) \ , \quad t \in [0,3] \medskip \\
x_{\tau}(t) = 1 \ , \ t \in [-\tau^1,0] \quad , \quad  u_{\tau}(t) = 0 \ , \ t \in [-\tau^2,0) \medskip \\
u_{\tau}(\cdot) \in L^{\infty}([-\tau^2,3],\mathbb{R}) \quad , \quad \tau = (\tau^1,\tau^2) = (1,2)
\end{cases}
\end{equation*}
Since no terminal conditions are imposed, this particular (\textbf{OCP})$_{\tau}$ can only have normal extremals. Then, the Hamiltonian is $H(t,x,y,p,u,v) = p y v - x^2 - u^2$ and the adjoint equation is $\dot{p}_{\tau}(t) = 2 x_{\tau}(t) - \mathds{1}_{[0,3-\tau^1]}(t) p_{\tau}(t+\tau^1) u_{\tau}(t-\tau^2+\tau^1)$, with $p_{\tau}(3) = 0$. Finally, we infer from the maximization condition (\ref{maxCond}), that optimal controls are given by $u_{\tau}(t) = \frac{1}{2} \mathds{1}_{[0,3-\tau^2]}(t) x_{\tau}(t+\tau^2-\tau^1) p_{\tau}(t+\tau^2)$. The paper \cite{gollmann2009optimal} shows that the optimal synthesis of (\textbf{OCP})$^1_{\tau}$ can be obtained analytically. In particular, one has
\begin{equation} \label{contEx1}
u_{2}(t) = \frac{\mathds{1}_{[0,1]}(t)}{e^2+1} (e^{t} - e^{2-t}) \quad , \quad
x_{1}(t) = \mathds{1}_{[0,2]}(t) + \frac{\mathds{1}_{[2,3]}(t)}{e^2+1} (e^{t-2} + e^{4-t})
\end{equation}
Considering Remark \ref{remarkExample}, we apply Algorithm \ref{alg} to solve (\textbf{OCP})$^1_{\tau}$ with $N = 1/h = 60$ Runge-Kutta time steps and a tolerance of $10^{-10}$ and $1500$ maximal iterations for \textit{hybrd} routine. Using a Simpson's rule, the optimal value $J^1(x_{\tau},u_{\tau}) = 2.76173$ is obtained in $20 \, \textnormal{ms}$ just in one iteration of the continuation scheme. Moreover, global errors in the sup norm between (\ref{contEx1}) and their numerical approximations respectively of $0.024 \, \%$ for the control and of $0.031 \, \%$ for the state are obtained. Figure \ref{figEx1} shows the optimal quantities for (\textbf{OCP})$^1_{\tau}$, its non-delayed version and an intermediate solution when $\tau = (0.5,1)$.
\begin{figure}[htpb] \label{figEx1}
\centering
\includegraphics[width=1.\textwidth]{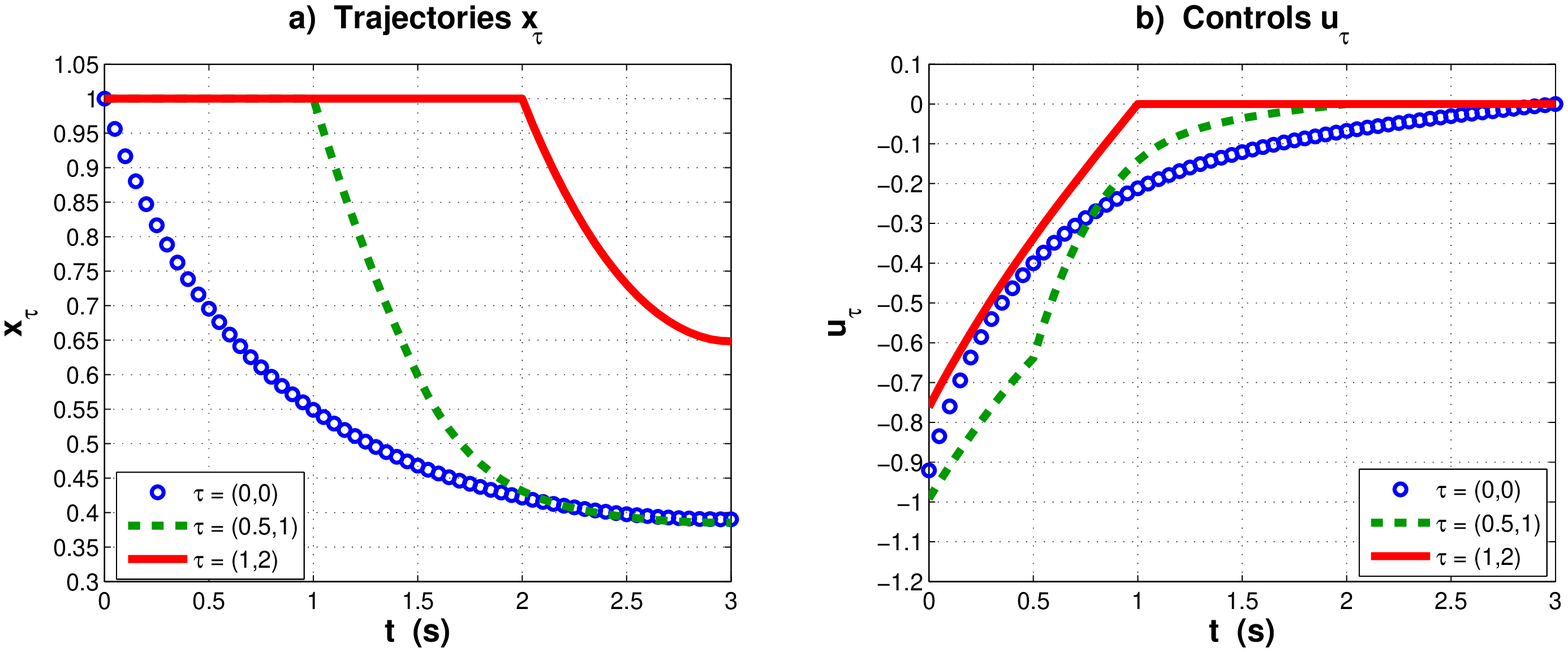}

\includegraphics[width=0.46\textwidth]{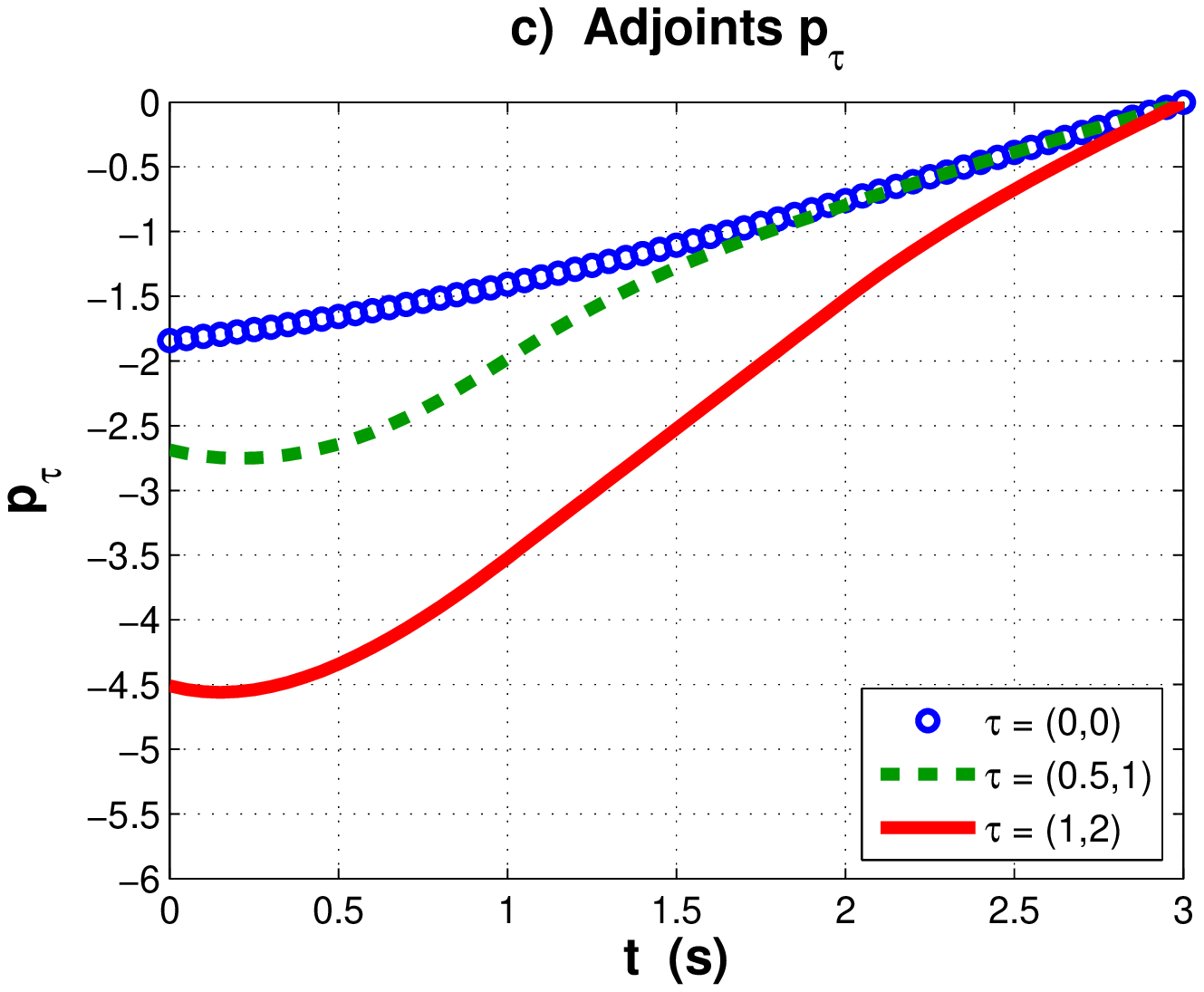}
\caption{a) Optimal states $x_{\tau}(\cdot)$, b) Optimal controls $u_{\tau}(\cdot)$, c) Optimal adjoints $p_{\tau}(\cdot)$ for (\textbf{OCP})$^1_{\tau}$.}
\end{figure}

\subsection{A Nonlinear Chemical Tank Reactor Model}

Let us consider a two-stage nonlinear continuous stirred tank reactor (CSTR) system with a first-order irreversible chemical reaction occurring in each tank. The system was studied by \cite{oh1976optimal} and successively by \cite{dadebo1992optimal} in the framework of the dynamic programming. This Optimal Control Problem with Delays (\textbf{OCP})$^{2}_{\tau}$ consists in minimizing the cost $J^2(x_{\tau},u_{\tau}) =\int_{0}^{T} \Big( \|x_{\tau}(t)\|^2 + 0.1\|u_{\tau}(t)\|^2 \Big) \; dt$ subject to
\begin{equation*}
\begin{cases}
\dot{x}^1_{\tau}(t) = 0.5 - x^1_{\tau}(t) - R_1(x^1_{\tau}(t),x^2_{\tau}(t)) \ , \ t \in [0,T] \medskip \\
\dot{x}^2_{\tau}(t) = R_1(x^1_{\tau}(t),x^2_{\tau}(t)) - (u^1_{\tau}(t) + 2)(x^2_{\tau}(t) + 0.25) \ , \ t \in [0,T] \medskip \\
\dot{x}^3_{\tau}(t) = x^1_{\tau}(t-\tau) - x^3_{\tau}(t) - R_2(x^3_{\tau}(t),x^4_{\tau}(t)) + 0.25 \ , \ t \in [0,T] \medskip \\
\dot{x}^4_{\tau}(t) =x^2_{\tau}(t-\tau) - 2 x^4_{\tau}(t) -u^2_{\tau}(t) (x^4_{\tau}(t) + 0.25) + R_2(x^3_{\tau}(t),x^4_{\tau}(t)) \medskip \\
\hspace{26pt} - 0.25 \ , \ t \in [0,T] \medskip \\
x^1_{\tau}(t) = 0.15 \ , \ x^2_{\tau}(t) = -0.03 \ , \ t \in [-\tau,0] \quad , \quad x^3_{\tau}(0) = 0.1 \ , \ x^4_{\tau}(0) = 0 \medskip \\
u^1_{\tau}(\cdot) \ , \ u^2_{\tau}(\cdot) \in L^{\infty}([0,T],\mathbb{R})
\end{cases}
\end{equation*}
where, now, we have a fixed scalar delay $\tau$ which is chosen in the interval $[0,0.8]$ and acts on the state only, the final time $T = 2$ is fixed and functions $R_1$, $R_2$ are given by $R_1(x,y) = (x + 0.5) \textnormal{exp}\bigg( \frac{25 y}{y + 2} \bigg)$, $R_2(x,y) = (x + 0.25) \textnormal{exp}\bigg( \frac{25 y}{y + 2} \bigg)$. Since no terminal conditions are imposed, (\textbf{OCP})$^{2}_{\tau}$ have only normal extremals. The Hamiltonian is
$$
H = p^1 \Big(0.5 - x^1 - R_1(x^1,x^2)\Big) + p^2 \Big(R_1(x^1,x^2) - (u^1+2)(x^2+0.25)\Big)
$$
$$
+ p^3 \Big(y^1 - x^3 - R_2(x^3,x^4) + 0.25\Big) + p^4 \Big(y^2 - 2 x^4 - u^2(x^4 + 0.25) + R_2(x^3,x^4) - 0.25\Big)
$$
$$
- \Big( (x^1)^2 + (x^2)^2 + (x^3)^2 + (x^4)^2 + 0.1 (u^1)^2 + 0.1 (u^2)^2 \Big)
$$
Thus, the adjoint equations take the following form
\begin{equation*}
\begin{cases}
\displaystyle \dot{p}^1_{\tau}(t) = p^1_{\tau}(t) + \frac{\partial R_1}{\partial x}(x^1_{\tau}(t),x^2_{\tau}(t)) \Big( p^1_{\tau}(t) - p^2_{\tau}(t) \Big) + 2 x^1_{\tau}(t)  \medskip \\
\hspace{26pt} - \mathds{1}_{[0,T-\tau]}(t) p^3_{\tau}(t+\tau) \medskip \\
\displaystyle \dot{p}^2_{\tau}(t) = p^2_{\tau}(t) (u^1_{\tau}(t) + 2) + \frac{\partial R_1}{\partial y}(x^1_{\tau}(t),x^2_{\tau}(t)) \Big( p^1_{\tau}(t) - p^2_{\tau}(t) \Big) + 2 x^2_{\tau}(t) \medskip \\
\hspace{26pt} - \mathds{1}_{[0,T-\tau]}(t) p^4_{\tau}(t+\tau) \medskip \\
\displaystyle \dot{p}^3_{\tau}(t) = p^3_{\tau}(t) + \frac{\partial R_2}{\partial x}(x^3_{\tau}(t),x^4_{\tau}(t)) \Big( p^3_{\tau}(t) - p^4_{\tau}(t) \Big) + 2 x^3_{\tau}(t) \medskip \\
\displaystyle \dot{p}^4_{\tau}(t) = p^4_{\tau}(t) (u^2_{\tau} + 2) + \frac{\partial R_2}{\partial y}(x^3_{\tau}(t),x^4_{\tau}(t)) \Big( p^3_{\tau}(t) - p^4_{\tau}(t) \Big) + 2 x^4_{\tau}(t)
\end{cases}
\end{equation*}
with $p_{\tau}(T) = 0$. From the maximality condition (\ref{maxCond}), the optimal controls are given by $u^1_{\tau}(t) = -5 p^2_{\tau}(t) \Big( x^2_{\tau}(t) + 0.25 \Big)$, $u^2_{\tau}(t) = -5 p^4_{\tau}(t) \Big( x^4_{\tau}(t) + 0.25 \Big)$.

\begingroup
\begin{table}[ht]
\centering
\begin{tabular}{c c c c c c c}
\hline\hline
$\tau$ & 0 & 0.05 & 0.1 & 0.2 & 0.4 \\ [0.5ex]
\hline
$J^2(x_{\tau},u_{\tau})$ & 0.02248 & 0.02282 & 0.02313 & 0.02370 & 0.02459 \\ [1ex]
\hline
N. of Continuation Iterations & 1 & 1 & 1 & 1 & 3 \\ [1ex]
\hline
Computational Time (s) & 0.200 & 0.240 & 0.450 & 0.820 & 2.120 \\ [1ex]
\hline\hline
\end{tabular}

\vspace{20pt}

\begin{tabular}{c c c c c c c}
\hline\hline
$\tau$ & 0.6 & 0.8 & 1 & 1.2 & 1.5 \\ [0.5ex]
\hline
$J^2(x_{\tau},u_{\tau})$ & 0.02516 & 0.02547 & 0.02556 & 0.02549 & 0.02527 \\ [1ex]
\hline
N. of Continuation Iterations & 5 & 5 & 5 & 3 & 6 \\ [1ex]
\hline
Computational Time (s) & 2.860 & 3.260 & 3.070 & 2.150 & 4.220 \\ [1ex]
\hline\hline
\end{tabular}

\caption{Optimal values of (\textbf{OCP})$^{2}_{\tau}$ for different delays $\tau$.}
\label{tableCostEx2}
\end{table}
\endgroup
We test several different delays $\tau$, as it was done in \cite{oh1976optimal}. We use $N = 1/h = 50$ Runge-Kutta time steps, a tolerance of $10^{-10}$ and $1500$ maximal iterations for \textit{hybrd} routine and a Simpson's rule for the cost. Our results are reported in Table 1.

Clearly these values are comparable with the ones obtained by \cite{oh1976optimal} and \cite{dadebo1992optimal}. Moreover, our continuation scheme finds a solution also for larger delays i.e., $\tau \in [0.8,1.5]$ (see Table 1). As expected, the more the delay grows the larger the number of iterations of the continuation method is. In order to check our results, we implemented also these problems within an AMPL framework (IPOPT solver with an explicit forward Euler scheme and N = 100000) and the optimal values provided by the direct method are the same. Nevertheless, for $10000 \le N \le 100000$, AMPL takes between 4.4 and 43.7 seconds to find the optimal solution, while our simulations take at most 4.3 s. Some optimal quantities obtained using our method are shown in Figure \ref{figEx2}.
\begin{figure}[htpb] \label{figEx2}
\centering
\includegraphics[width=1.\textwidth]{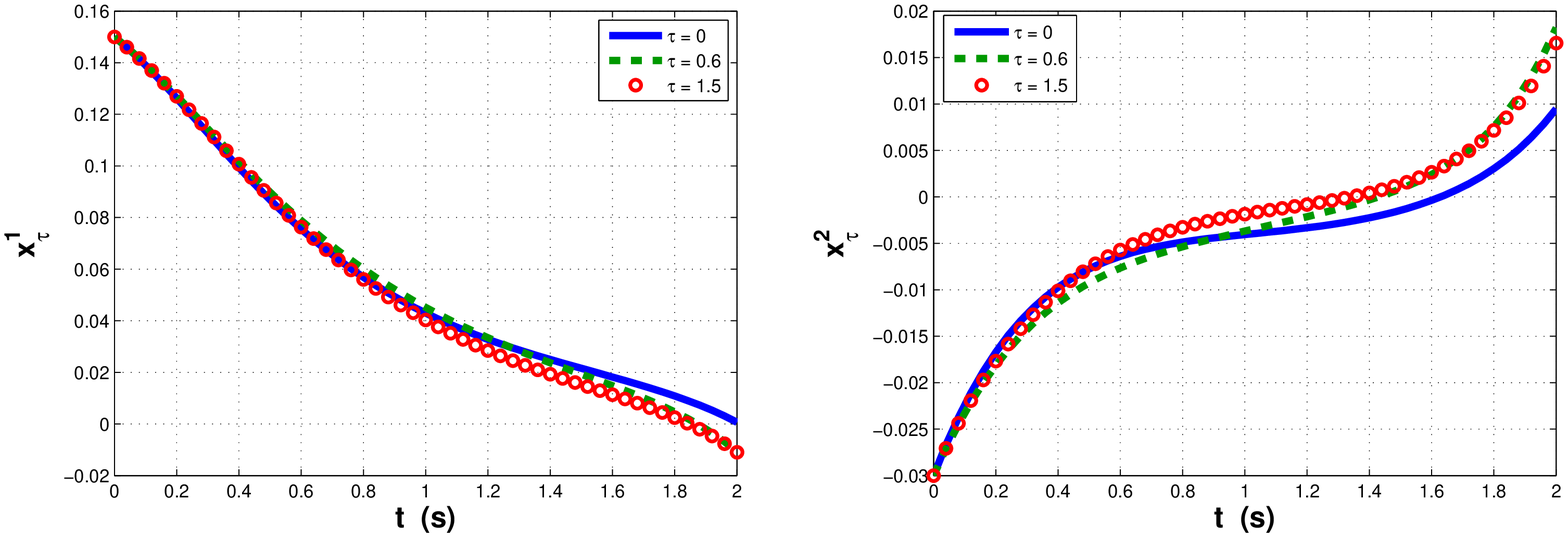}

\includegraphics[width=1.\textwidth]{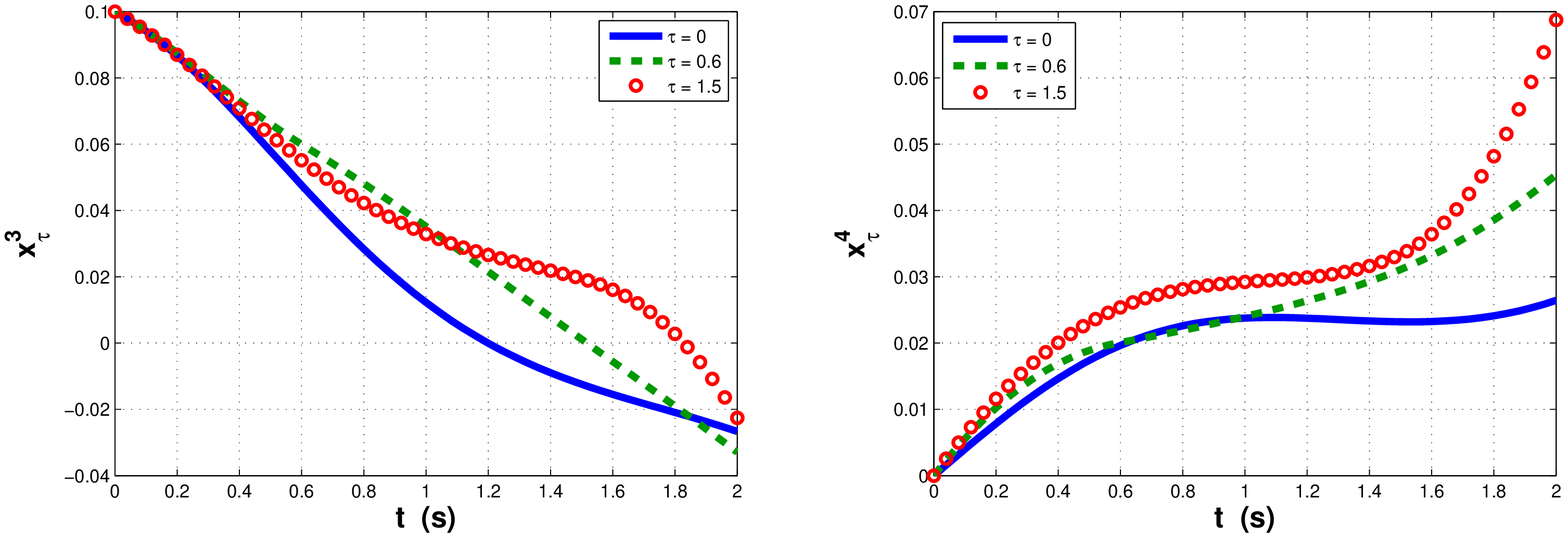}

\includegraphics[width=1.\textwidth]{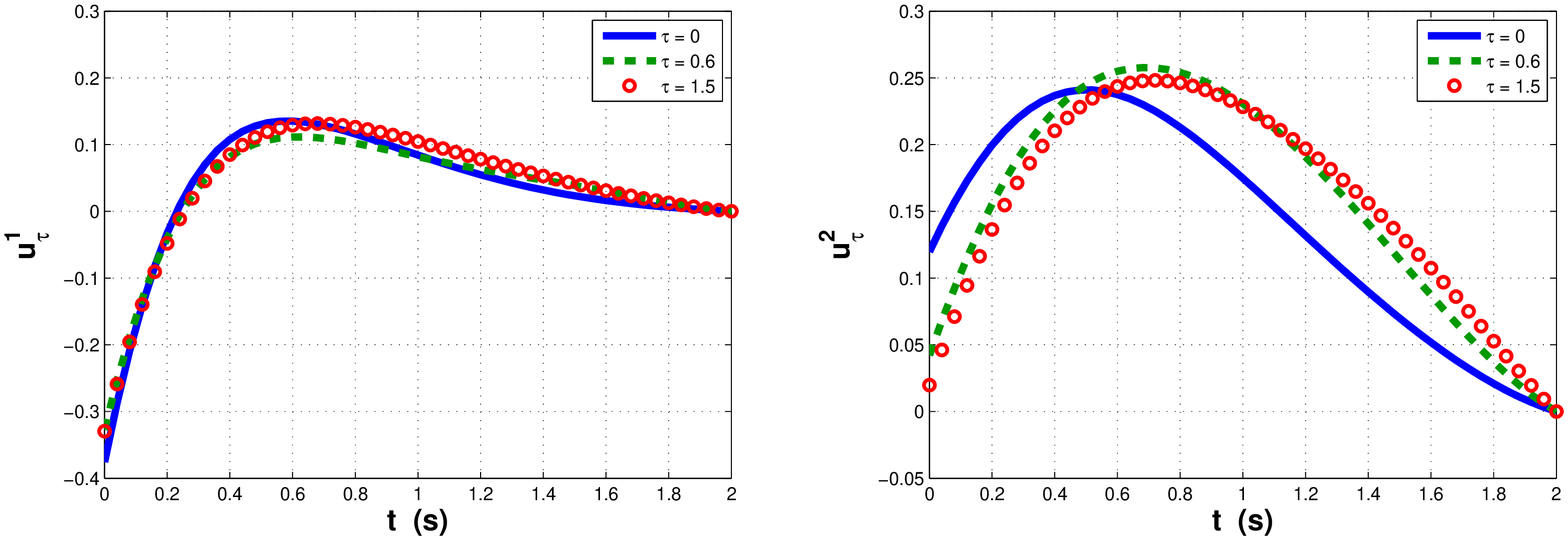}
\caption{Optimal quantities of (\textbf{OCP})$^{2}_{\tau}$ for different delays $\tau$.}
\end{figure}

\section{Conclusions}

In this paper we propose a new methodology to solve optimal control problems with delays by means of a shooting method combined with continuation on the delay. It is known that necessary conditions applied to this kind of problems lead to differential-difference boundary value problems which are strongly complex to solve. We overcome this difficulty by introducing the delay step by step exploiting homotopy procedures. Under appropriate assumptions, we provide a crucial convergence result that justifies rigorously the proposed approach. Moreover, the numerical simulations confirm the validity of the procedure for nontrivial applications.

Future works focus on applying this strategy to compute in \textit{real-time} optimal solutions in the context of \textit{aeronautic motion planning}. In several situations, when simulating launch vehicle systems, fast optimal solutions must be provided taking into account phenomena like the \textit{non-minimum phase} (see, e.g. \cite{balas2012adaptive}). Since this phenomenon can be approximated by delays, one is led to manage a fast resolution of optimal control problems with delays. Then, indirect methods are preferred and the proposed approach is a good candidate for a real-time processing. Moreover, in this constext, it is shown (see, e.g. \cite{myIFAC2017}) that the solution of the non-delayed version of the problem can be rapidly provided by combining shooting methods again with homotopy procedures. This fixes the flaw concerning the initialization of (\textbf{OCP}) in the launch vehicle context.

\newpage

\appendix
\section{Proof of Theorem \ref{theoMain}} \label{appProof}
This final section is devoted to the proof of Theorem \ref{theoMain} and is organized in two main parts.

First, the fundamental steps of the classical proof of the PMP with constant delays in the state and the control are recalled. Basically, this proof lies on the development of some ad hoc variations of the optimal control and it is interesting to note that such approach is barely presented in the litterature. This proof is important because it introduces tools and techniques which are the basis for the proof of our main result.

Afterwards, the main result is proved. The questions of the existence of delayed controls and trajectories and their convergence are addressed firstly. As final result, the convergence of delayed adjoint vectors is proposed. This nontrivial result requires an accurate analysis of the convergence of Pontryagin cones which is accomplished using a conic version of the implicit function theorem (see, e.g. \cite{silva2010smooth,haberkorn2011convergence}).
\subsection{Delayed Maximum Principle}
In this section we sketch the proof of the Maximum Principle for (\textbf{OCP})$_{\tau}$ using needle-like variations. For this, we do not rely on the assumptions of Theorem \ref{theoMain}, giving the result for a large class of control systems with delays. It is interesting to note that our reasoning is not affected for problem with free final time $T_{\tau}$. Indeed, we do not employ the well known reduction to a fixed final time problem, but rather, we modify conveniently the Pontryagin cone to keep track of the free variable $T_{\tau}$ (as done in \cite{kharatishvili1961maximum}).
\subsubsection{Preliminaries}
For every $\tau = (\tau^1,\tau^2) \in [0,\Delta]^2$, consider (\textbf{OCP})$_{\tau}$. For every positive final time $T$, introduce the instantaneous cost function $x^0_{\tau}(\cdot)$, defined on $[-\Delta,T]$ and solution of
\begin{equation*}
\begin{cases}
\dot{x}^0_{\tau}(t) = f^0(t,x_{\tau}(t),x_{\tau}(t-\tau^1),u_{\tau}(t),u_{\tau}(t-\tau^2)) \ , \quad t \in [0,T] \medskip \\
x^0_{\tau}(t) = 0 \ , \ t \in [-\Delta,0]
\end{cases}
\end{equation*}
such that the cost $C_{T}(\tau,u_{\tau}(\cdot))$ of the initial trajectory $x_{\tau}(\cdot)$ is $C_{T}(\tau,u_{\tau}(\cdot)) = x^0_{\tau}(T)$. The extended state $\tilde x \in \mathbb{R}^{n+1}$ is defined by $\tilde x = (x,x^0)$, and the extended dynamics by $\tilde f(t,\tilde x,\tilde y,u,v) = (f(t,x,y,u,v),f^0(t,x,y,u,v))$. We will often denote $\tilde f(t,x,y,u,v) = \tilde f(t,\tilde x,\tilde y,u,v)$. Then, consider the extended delayed control system in $\mathbb{R}^{n+1}$,
\begin{equation} \label{dynDelayAug}
\begin{cases}
\dot{\tilde{x}}_{\tau}(t) = f(t,\tilde x_{\tau}(t),\tilde x_{\tau}(t-\tau^1),u_{\tau}(t),u_{\tau}(t-\tau^2)) \ , \quad t \in [0,T] \medskip \\
\tilde x_{\tau}(t) = (\phi^1(t),0) \ , \ t \in [-\Delta,0] \quad , \quad  u_{\tau}(t) = \phi^2(t) \ , \ t \in [-\Delta,0) \medskip \\
u_{\tau}(\cdot) \in L^{\infty}([-\Delta,T],\Omega)
\end{cases}
\end{equation}

Admissible controls for system (\ref{dynDelayAug}) are defined in the same way as we defined admissible controls for (\ref{dynDelay}), and we denote $\tilde{\mathcal{U}}^{\tau}_{T,\mathbb{R}^m}$ the set of all admissible controls of (\ref{dynDelayAug}) taking their values in $\mathbb{R}^m$ while $\tilde{\mathcal{U}}^{\tau}_{T,\Omega}$ denotes the set of all admissible controls of (\ref{dynDelayAug}) with values in $\Omega$. Then, $\tilde{\mathcal{U}}^{\tau}_{\mathbb{R}^m} := \cup_{T > 0} \tilde{\mathcal{U}}^{\tau}_{T,\mathbb{R}^m}$ and $\tilde{\mathcal{U}}^{\tau}_{\Omega} := \cup_{T > 0} \tilde{\mathcal{U}}^{\tau}_{T,\Omega}$.

For every $u_{\tau}(\cdot) \in \tilde{\mathcal{U}}^{\tau}_{T,\mathbb{R}^m}$, the extended \textit{end-point mapping} $\tilde E_{\tau}$ is then defined by $\tilde E_{\tau}(T,u_{\tau}(\cdot)) = \tilde x_{\tau}(T)$. As standard facts, the set $\tilde{\mathcal{U}}^{\tau}_{T,\mathbb{R}^m}$, endowed with the standard topology of $L^{\infty}([-\Delta,T],\mathbb{R}^m)$, is open and the end-point mapping is $C^1$ on $\tilde{\mathcal{U}}^{\tau}_{T,\mathbb{R}^m}$.

For every $\tau = (\tau^1,\tau^2) \in [0,\Delta]^2$ and every $t \ge 0$, define the extended \textit{accessible set} $\tilde{\mathcal{A}}_{\tau,\Omega}(t)$ as the image of the mapping $\tilde E_{\tau}(t,\cdot) : \tilde{\mathcal{U}}^{\tau}_{t,\Omega} \rightarrow \mathbb{R}^{n+1}$, with the agreement $\tilde{\mathcal{A}}_{\tau,\Omega}(0) = \{ (\phi^1(0),0) \}$. The following classical fact is at the basis of the proof of the Maximum Principle.
\begin{lemma} \label{interior}
Given a couple of delays $\tau = (\tau^1,\tau^2) \in [0,\Delta]^2$, let $(x_{\tau}(\cdot),u_{\tau}(\cdot))$ be a solution of (\textbf{OCP})$_{\tau}$ defined on $[-\Delta,T_{\tau}]$. Then, the point $\tilde x_{\tau}(T_{\tau})$ belongs to the boundary of the set $\tilde{\mathcal{A}}_{\tau,\Omega}(T_{\tau})$.
\end{lemma}
\subsubsection{Needle-Like Variations and Pontryagin Cone} \label{sectNeedle}
In what follows we consider (\textbf{OCP})$_{\tau}$ where the final time $T_{\tau}$ is free; all the proposed results can be easily adapted when $T_{\tau}$ is fixed. Moreover, we suppose that $T_{\tau}$ is a Lebesgue point for the optimal control $u_{\tau}(\cdot)$ of (\textbf{OCP})$_{\tau}$ and of $u_{\tau}(\cdot-\tau_2)$. If it is not, we can extend all the conclusions by using closure points of $T_{\tau}$ (in the same way as in \cite{lee1967foundations,gamkrelidze2013principles}).

Given a couple of delays $\tau = (\tau^1,\tau^2) \in [0,\Delta]^2$, let $(x_{\tau}(\cdot),u_{\tau}(\cdot))$ be a solution of (\textbf{OCP})$_{\tau}$. Without loss of generality, we extend $u_{\tau}(\cdot)$ by some constant vector of $\Omega$ in $[T_{\tau},T_{\tau}+\tau^2]$. Let $p \ge 1$ be an integer and consider $0 < t_1 < \dots < t_p < T_{\tau}$ Lebesgue points respectively of $u_{\tau}(\cdot)$, $u_{\tau}(\cdot-\tau^2)$ and of $u_{\tau}(\cdot+\tau^2)$. Chosing $p$ arbitrary values $u_i \in \Omega$, for every $\eta_i > 0$ such that $t_i + \eta_i < T_{\tau}$, the \textit{needle-like variation} $\pi = \{ t_1,\dots,t_p,\eta_1,\dots,\eta_p,u_1,\dots,u_p \}$ of the control $u_{\tau}(\cdot)$ is defined by
$$
u^{\pi}_{\tau}(t) := \begin{cases} u_i & t \in [t_i,t_i + \eta_i), \\ u_{\tau}(t) & \textnormal{otherwise}. \end{cases}
$$
Control $u^{\pi}_{\tau}(\cdot)$ takes its values in $\Omega$ and, by continuity w.r.t. initial data, it is clear that, if every $\eta_i > 0$ is small enough, then $u^{\pi}_{\tau}(\cdot) \in \tilde{\mathcal{U}}^{\tau}_{\Omega}$. Moreover, the trajectory $\tilde x^{\pi}_{\tau}(\cdot)$ solution of (\ref{dynDelayAug}) with control $u^{\pi}_{\tau}(\cdot)$ converges uniformly to $\tilde x_{\tau}(\cdot)$ whenever $\| (\eta_1,\dots,\eta_p) \| \rightarrow 0$.

For every $z \in \Omega$ and appropriate Lebesgue point $s$, we define the vectors
\begin{eqnarray} \label{omegaMinus}
\omega^-_z(s) &:=& \tilde f(s,\tilde x_{\tau}(s),\tilde x_{\tau}(s-\tau^1),z,u_{\tau}(s-\tau^2)) \\
&-& \tilde f(s,\tilde x_{\tau}(s),\tilde x_{\tau}(s-\tau^1),u_{\tau}(s),u_{\tau}(s-\tau^2)) \nonumber
\end{eqnarray}
\begin{eqnarray} \label{omegaPlus}
\omega^+_z(s) &:=& \tilde f(s+\tau^2,\tilde x_{\tau}(s+\tau^2),\tilde x_{\tau}(s+\tau^2-\tau^1),u_{\tau}(s+\tau^2),z) \\
&-& \tilde f(s+\tau^2,\tilde x_{\tau}(s+\tau^2),\tilde x_{\tau}(s+\tau^2-\tau^1),u_{\tau}(s+\tau^2),u_{\tau}(s)) \nonumber
\end{eqnarray}
and, given $\xi \in \mathbb{R}^{n+1}$ and $s \in (0,T_{\tau})$, we denote $\tilde v^{\tau}_{s,\xi}(\cdot)$ the solution of
\begin{eqnarray} \label{dynVariation}
\begin{cases}
\dot{\psi}(t) &= \displaystyle \frac{\partial \tilde f}{\partial x}(t,\tilde x_{\tau}(t),\tilde x_{\tau}(t-\tau^1),u_{\tau}(t),u_{\tau}(t-\tau^2)) \psi(t) \medskip \\
&+ \displaystyle \frac{\partial \tilde f}{\partial y}(t,\tilde x_{\tau}(t),\tilde x_{\tau}(t-\tau^1),u_{\tau}(t),u_{\tau}(t-\tau^2)) \psi(t-\tau^1) \medskip \\
\psi(s) &= \xi \quad , \quad \psi(t) = 0 \ , \ t \in (s-\tau^1,s)
\end{cases}
\end{eqnarray}
Vectors $\tilde v^{\tau}_{s,\xi}(\cdot)$ are called \textit{variations vectors}. The first crucial part of the proof of the Maximum Principle consists of the following result.
\begin{lemma} \label{needleLike}
Let $\delta \in \mathbb{R}$ small enough. Then, for every Lebesgue point $t \le T_{\tau}$ of $u_{\tau}(\cdot)$ and $u_{\tau}(\cdot-\tau^2)$, the following expression holds
\begin{eqnarray*}
\tilde x^{\pi}_{\tau}(t + \delta) &=& \tilde x_{\tau}(t) + \delta \tilde f(t,\tilde x_{\tau}(t),\tilde x_{\tau}(t-\tau^1),u_{\tau}(t),u_{\tau}(t-\tau^2)) \medskip \\
&+& \sum_{i=1}^{p} \eta_i \Big( \tilde v^{\tau}_{t_i,\omega^-_{u_i}(t_i)}(t) + \tilde v^{\tau}_{t_i+\tau^2,\omega^+_{u_i}(t_i)}(t) \Big) + o \Big( \delta + \sum_{i=1}^{p} \eta_i \Big)
\end{eqnarray*}
\end{lemma}
\begin{proof} This is done by induction on Gronwall inequalities as in the classical proof of the PMP. Unlike the classical approach, additional $p$ terms of type
\begingroup
\small
$$
\int_{t_i+\tau^2}^{t_i+\tau^2+\eta_i} \Big[ \tilde f(s,\tilde x^{\pi}_{\tau}(s),\tilde x^{\pi}_{\tau}(s-\tau^1),u_{\tau}(s),u_i) - f(s,\tilde x_{\tau}(s),\tilde x_{\tau}(s-\tau^1),u_{\tau}(s),u_{\tau}(s-\tau^2)) \Big] \; ds
$$
\endgroup
appear, producing further discontinuities in the derivative of the trajectory at $t_i+\tau^2$.
\end{proof}

It is useful to denote $\tilde w^{\tau}_{s,z}(t) := \tilde v^{\tau}_{s,\omega^-_z(s)}(t) + \tilde v^{\tau}_{s+\tau^2,\omega^+_z(s)}(t)$.
\begin{definition} \label{defCone}
For every $t \in (0,T_{\tau}]$, the first Pontryagin cone $\tilde K^{\tau}(t) \subseteq \mathbb{R}^{n+1}$ at $\tilde x_{\tau}(t)$ for the extended system is defined as the closed convex cone containing vectors $\tilde w^{\tau}_{s,z}(t)$ where $z \in \Omega$ and $0 < s < t$ is a Lebesgue point of $u_{\tau}(\cdot)$, $u_{\tau}(\cdot-\tau^2)$ and of $u_{\tau}(\cdot+\tau^2)$. The augmented first Pontryagin cone $\tilde K^{\tau}_1(t) \subseteq \mathbb{R}^{n+1}$ at $\tilde x_{\tau}(t)$ for the extended system is defined as the closed convex cone containing $\tilde f(t,\tilde x_{\tau}(t),\tilde x_{\tau}(t-\tau^1),u_{\tau}(t),u_{\tau}(t-\tau^2))$, $-\tilde f(t,\tilde x_{\tau}(t),\tilde x_{\tau}(t-\tau^1),u_{\tau}(t),u_{\tau}(t-\tau^2))$ and vectors $\tilde w^{\tau}_{s,z}(t)$ where $z \in \Omega$ and $0 < s < t$ is a Lebesgue point of $u_{\tau}(\cdot)$, $u_{\tau}(\cdot-\tau^2)$ and of $u_{\tau}(\cdot+\tau^2)$.

The first Pontryagin cone $K^{\tau}(t) \subseteq \mathbb{R}^{n}$ and the augmented first Pontryagin cone $K^{\tau}_1(t) \subseteq \mathbb{R}^{n}$ at $x_{\tau}(t)$ for the initial system are defined similarly, considering the initial dynamics $f$ instead of the extended dynamics $\tilde f$. Obviously, $K^{\tau}(t)$ and $K^{\tau}_1(t)$ are the projections on $\mathbb{R}^n$ respectively of $\tilde K^{\tau}(t)$ and $\tilde K^{\tau}_1(t)$.
\end{definition}
\subsubsection{The Maximum Principle}
The proof of the Maximum Principle using needle-like variations argues by contradiction using the definition of first Pontryagin cones and the following lemma, whose proof can be found in \cite{agrachev2013control}.
\begin{lemma}[Conic Implicit Function Theorem] \label{iFT}
Let $C \subseteq \mathbb{R}^m$ be a convex subset with non empty interior, of vertex 0, and $F : C \rightarrow \mathbb{R}^n$ be a Lipschitzian mapping such that $F(0) = 0$ and $F$ is G\^{a}teaux differentiable at 0. Assume that $dF(0)\cdot \textnormal{Cone}(C) = \mathbb{R}^n$, where $\textnormal{Cone}(C)$ stands for the convex cone generated by elements of $C$. Then, 0 belongs to the interior of $F(\mathcal{V} \cap C)$, for every neighborhood $\mathcal{V}$ of 0 in $\mathbb{R}^m$.
\end{lemma}

Now, consider any integer $p \ge 1$ and a positive real number $\varepsilon_p > 0$. Define
$$
G^{\tau}_p : B_{\varepsilon_p}(0) \cap \mathbb{R} \times  \mathbb{R}^p_+ \rightarrow \mathbb{R}^{n+1} : (\delta,\eta_1,\dots,\eta_p) \mapsto \tilde x^{\pi}_{\tau}(T_{\tau} + \delta) - \tilde x_{\tau}(T_{\tau})
$$
where $\pi$ is a variation and $\varepsilon_p$ is small enough such that $G^{\tau}_p$ is well defined. We have:
\begin{itemize}
\item $G^{\tau}_p(0) = 0$ and $G^{\tau}_p$ is Lipschitz continuous;
\item $G^{\tau}_p$ is G\^{a}teaux differentiable at 0 (from Lemma \ref{needleLike});
\end{itemize}
Note that, the Lipschitz behavior of $G^{\tau}_p$ is proved by a recursive use of needle-like variations at $t_i - \varepsilon$, $1 \le i \le p$ (for $\varepsilon$ small enough), Lebesgue points of $u(\cdot)$. Since $t_i - \varepsilon$ are Lebesgue points of $u(\cdot)$ only for almost every $\varepsilon$, the recursive use of Lemma \ref{needleLike} can be done only almost everywhere. The conclusion follows from the continuity of $G^{\tau}_p$. The Maximum Principle is then established as follows.

Suppose, by contradiction, that $\tilde K^{\tau}_1(T_{\tau})$ coincides with $\mathbb{R}^{n+1}$. Then, by definition, there would exist an integer $p \ge 1$, a variation and a positive real number $\varepsilon_p > 0$ such that $d G^{\tau}_p(0)\cdot(\mathbb{R} \times \mathbb{R}^p_+) = \mathbb{R}^{n+1}$, and then Lemma \ref{iFT} would imply that the point $\tilde x_{\tau}(T_{\tau})$ belongs to the interior of the accessible set $\tilde{\mathcal{A}}_{\tau,\Omega}(T_{\tau})$, which contradicts Lemma \ref{interior}. Therefore:
\begin{lemma}
There exists $\tilde \psi_{\tau} \in \mathbb{R}^{n+1} \setminus \{ 0 \}$ (\textit{Lagrange multiplier}) such that
\begin{eqnarray*}
&\langle \tilde \psi_{\tau} , \tilde f(T_{\tau},\tilde x_{\tau}(T_{\tau}),\tilde x_{\tau}(T_{\tau}-\tau^1),u_{\tau}(T_{\tau}),u_{\tau}(T_{\tau}-\tau^2)) \rangle = 0 \medskip \\
&\langle \tilde \psi_{\tau} , \tilde v_{\tau} \rangle \le 0 \quad , \quad \forall \ \tilde v_{\tau} \in \tilde K^{\tau}(T_{\tau})
\end{eqnarray*}
\end{lemma}
These relations permit to derive (in the usual way, see, e.g. \cite{pontryagin1987mathematical}) the necessary conditions (\ref{dynDual})-(\ref{freeT}) given in the first section. The relation with the above Lagrange multiplier $\tilde \psi_{\tau} = (\psi_{\tau},\psi^0_{\tau})$ is that the adjoint vector is constructed so that $p_{\tau}(T_{\tau}) = \psi_{\tau}$, $p^0_{\tau} = \psi^0_{\tau}$.

From the previous considerations, the following useful result follows
\begin{lemma} \label{emmanuelPlane}
Suppose that the final time $T_{\tau}$ of (\textbf{OCP})$_{\tau}$ is free. For the optimal trajectory $x_{\tau}(\cdot)$ the following statements are equivalent:
\begin{itemize}
\item The trajectory $x_{\tau}(\cdot)$ has an unique extremal lift $(x_{\tau}(\cdot),p_{\tau}(\cdot),p^0_{\tau},u_{\tau}(\cdot))$ whose adjoint $(p_{\tau}(\cdot),p^0_{\tau})$ is unique up to a multiplicative scalar, which is moreover normal i.e., $p^0_{\tau} < 0$;
\item $\tilde K^{\tau}_1(T_{\tau})$ is a half-space of $\mathbb{R}^{n+1}$ and $K^{\tau}_1(T_{\tau}) = \mathbb{R}^n$.
\end{itemize}
\end{lemma}
\subsection{Proof of the Main Result}
The proof exploits the results introduced previously. Unfortunately, Lemma \ref{iFT} does not take into account the dependence w.r.t. the delay $\tau$ and a more general version of this lemma, depending on parameters, must be introduced.

As pointed out previously, when the delay varies, Lemma \ref{needleLike} can be applied only to almost every $\tau$. This obliges to introduce a notion of conic implicit function theorem depending on parameters on dense subsets.

A function $f : C \subseteq \mathbb{R}^n \rightarrow \mathbb{R}^m$ is said \textit{almost everywhere strictly differentiable at some point} $x_0 \in C$ whenever there exists a linear continuous mapping $df(x_0) : \mathbb{R}^n \rightarrow \mathbb{R}^m$, such that
$$
f(y) - f(x) = df(x_0)\cdot(y - x) + \| y - x \|g(x,y)
$$
for almost every $x,y \in C$, where $g(x,y) \rightarrow 0$ as $\| x - x_0 \| + \| y - x_0 \| \xrightarrow{a.e.} 0$.

\begin{lemma} \label{iFTP}
Let $C \subseteq \mathbb{R}^m$ be an open convex subset with non empty interior, of vertex 0, and $F : \mathbb{R}^2_+ \times C \rightarrow \mathbb{R}^n : (\varepsilon^1,\varepsilon^2,x) \mapsto F(\varepsilon^1,\varepsilon^2,x)$ be a continuous mapping satisfying the following assumptions:
\begin{itemize}
\item $F(0,0,0) = 0;$
\item For almost every $\varepsilon^1,\varepsilon^2 \ge 0$, $F$ is almost everywhere strictly differentiable w.r.t. $x$ at 0, and $\displaystyle \frac{\partial F}{\partial x}(\varepsilon^1,\varepsilon^2,0)$ is approximately continuous w.r.t. $(\varepsilon^1,\varepsilon^2)$;
\item $\displaystyle \frac{\partial F}{\partial x}(0,0,0)\cdot\textnormal{Cone}(C) = \mathbb{R}^n$.
\end{itemize}
Then, there exists $\varepsilon_0 > 0$, a neighborhood $V$ of 0 in $\mathbb{R}^n$, and a continuous function $h : [0,\varepsilon_0)^2 \times V \rightarrow \mathbb{R}^m$ with values in $C$, such that
$$
F(\varepsilon^1,\varepsilon^2,h(\varepsilon^1,\varepsilon^2,y)) = y
$$
for every $(\varepsilon^1,\varepsilon^2) \in [0,\varepsilon_0)^2$ and every $y \in V$.
\end{lemma}
The proof of Lemma \ref{iFTP} is done in Appendix \ref{appConic}.

From now on, assume that assumptions $(A_1)$, $(A_2)$, $(A_3)$ and $(A_4)$ hold and we summon the other assumptions of Theorem \ref{theoMain} when they are needed. Moreover, in the sequel $(x(\cdot),u(\cdot))$ denotes the (unique) solution of (\textbf{OCP}) and we assume that its related final time $T$ is a Lebesgue point of $u(\cdot)$ (if not, we refer the reader to the approach proposed by \cite{gamkrelidze2013principles}) and it is free (otherwise, the proof is similar but simpler). We divide the proof in several main subparts.
\subsubsection{Controllability} \label{controllability}
For any integer $p \ge 1$, fix $0<t_1<\dots<t_p<T$ Lebesgue points of control $u(\cdot)$ and $p$ arbitrary values $u_i \in \Omega$. We denote $v|_n$ the first $n$ coordinates of a vector $v \in \mathbb{R}^{n+1}$. For an appropriate small positive real number $\varepsilon_p > 0$, denoting by $\tilde x_{(\varepsilon^1,\varepsilon^2)}(\cdot)$ the trajectory solution of (\ref{dynDelayAug}) with control $u_{(\varepsilon^1,\varepsilon^2)}(\cdot) = u(\cdot)$, we define the mapping
$$
\Gamma : B_{\varepsilon_p}(0) \cap (\mathbb{R}^2_+ \times \mathbb{R} \times \mathbb{R}^p_+) \rightarrow \mathbb{R}^n : (\varepsilon^1,\varepsilon^2,\delta,\eta_1,\dots,\eta_p) \mapsto [ \tilde x^{\pi}_{(\varepsilon^1,\varepsilon^2)}(T + \delta) - \tilde x(T) ]|_n
$$
Thanks to Assumption $(A_2)$ and by continuity w.r.t. initial data, $\Gamma$ is well defined and continuous. Moreover, $\Gamma(0,\dots,0) = 0$ and we note that
$$
\Gamma(\varepsilon^1,\varepsilon^2,\delta,\eta_1,\dots,\eta_p) = [ \tilde x^{\pi}_{(\varepsilon^1,\varepsilon^2)}(T + \delta) - \tilde x_{(\varepsilon^1,\varepsilon^2)}(T) ]|_n + [ \tilde x_{(\varepsilon^1,\varepsilon^2)}(T) - \tilde x(T) ]|_n
$$
From Lemma \ref{needleLike} and a recursive use of needle-like variations, it follows that, for almost every $(\varepsilon^1,\varepsilon^2) \in \mathbb{R}^2_+$ small enough, $\Gamma$ is almost everywhere strictly differentiable w.r.t. $(\delta,\eta_1,\dots,\eta_p)$ at 0 and $\displaystyle \frac{\partial \Gamma}{\partial (\delta,\eta_1,\dots,\eta_p)}(\varepsilon^1,\varepsilon^2,0)$ is approximately continuous w.r.t. $(\varepsilon^1,\varepsilon^2)$. Finally, from Assumption $(A_3)$, the unique extremal lift of $x(\cdot)$ is normal, hence, it follows from Lemma \ref{emmanuelPlane} (applied to the non-delayed case) that $K^{0}_1(T) = \mathbb{R}^n$. Therefore, exploiting the particular form of the dynamics of (\textbf{OCP})$_{\tau}$, there exist a real number $\delta$, an integer $p \le 1$ and a variation $\pi = \{ t_1,\dots,t_p,\eta_1,\dots,\eta_p,u_1,\dots,u_p \}$ such that the associated mapping $\Gamma$ satisfies
$$
\frac{\partial \Gamma}{\partial (\delta,\eta_1,\dots,\eta_p)}(0,0,0)\cdot(\mathbb{R} \times \mathbb{R}^p_+) = K^{0}_1(T) = \mathbb{R}^n \quad .
$$

Then, Lemma \ref{iFTP} implies that there exists $\tau_0 > 0$ such that, for every $\tau = (\tau^1,\tau^2) \in [0,\tau_0)^2$, there exist a real number $\delta_{\tau}$ and positive real numbers $\eta^{\tau}_1,\dots,\eta^{\tau}_p$ such that $\Gamma(\tau^1,\tau^2,\delta_{\tau},\eta^{\tau}_1,\dots,\eta^{\tau}_p) = 0$. Moreover, $\delta_{\tau}$ is small whenever $\tau_0$ is small enough. From Assumption $(A_4)$, it follows that the subset $M$ is reachable for the control system (\ref{dynDelay}), in a final time $T_{\tau} \in [0,b]$, using control $u^{\pi}_{(\tau^1,\tau^2)}(\cdot) \in L^{\infty}([0,T_{\tau}],\Omega)$.

We have proved that, for every $\tau = (\tau^1,\tau^2) \in (0,\tau_0)^2$, (\textbf{OCP})$_{\tau}$ is controllable. Note that the previous argument still holds for optimal control problems (\textbf{OCP})$_{\tau}$ with pure state delays $\tau = (\tau^1,0)$.
\subsubsection{Existence of Optimal Controls} \label{existenceSect}
Now, we focus on the existence of an optimal control for (\textbf{OCP})$_{\tau}$, for every  $\tau = (\tau^1,\tau^2) \in (0,\tau_0)^2$.

First, we focus on the case where $f$ and $f^0$ are affine in the two control variables. No other assumptions but $(A_1)$-$(A_4)$ are considered. Notice that, in this case, the existence can be achieved by using the arguments in \cite[Theorem 2]{myACC2017}. However, we prefer to develop the usual Filippov's scheme (following \cite{trelat2008controle}) to highlight the difficulty in applying this scheme to more general dynamical systems (see Remark \ref{remarkGuinn}).

Fix $\tau = (\tau^1,\tau^2) \in (0,\tau_0)^2$ and let
$$
\alpha = \inf_{u_{\tau}(\cdot) \in \mathcal{U}^{\tau}_{\Omega}} C_{T_{\tau}(u_{\tau})}(\tau,u_{\tau}(\cdot)) = \int_{0}^{T_{\tau}(u_{\tau})} f^0(t,x_{\tau}(t),x_{\tau}(t-\tau^1),u_{\tau}(t),u_{\tau}(t-\tau^2)) \; dt
$$
Consider now a minimizing sequence of trajectories $x_{\tau,k}(\cdot)$ associated to $u_{\tau,k}(\cdot)$ i.e., $C_{T_{\tau}(u_{\tau,k})}(\tau,u_{\tau,k}(\cdot)) \rightarrow \alpha$ when $k \rightarrow \infty$ and define $\tilde F_k(t) := \tilde f(t,x_{\tau,k}(t),x_{\tau,k}(t-\tau^1),u_{\tau,k}(t),u_{\tau,k}(t-\tau^2))$ for almost every $t \in [0,T_{\tau}(u_{\tau,k})]$. By Assumption $(A_4)$, we can extend $\tilde F_k(\cdot)$ by zero on $(T_{\tau}(u_{\tau,k}),b]$ so that $(\tilde F_k(\cdot))_{k \in \mathbb{N}}$ is bounded in $L^{\infty}([0,b],\mathbb{R}^{n+1})$. Then, up to some subsequence, $(\tilde F_k(\cdot))_{k \in \mathbb{N}}$ converges to some $\tilde F(\cdot) = (F(\cdot),F^0(\cdot)) \in L^{\infty}([0,b],\mathbb{R}^{n+1})$ for the weak star topology of $L^{\infty}$. Up to some subsequence, the sequence $(T_{\tau}(u_{\tau,k}))_{k \in \mathbb{N}}$ converges to some $\bar T \ge 0$. Now, we define
$$
y(t) := \phi^1(t) \mathds{1}_{[-\Delta,0)}(t) + \mathds{1}_{[0,\bar T]}(t) \bigg[ \phi^1(0) + \int_{0}^{t} F(s) \; ds \bigg]
$$
for every $t \in [-\Delta,\bar T]$. Clearly, $y(\cdot)$ is absolutely constinuous and, taking continuous extensions, $(x_{\tau,k}(\cdot))_{k \in \mathbb{N}}$ converges pointwise to $y(\cdot)$ within $[-\Delta,\bar T]$. Firstly, we show that $y(\cdot)$ comes from an admissible control in $\mathcal{U}^{\tau}_{\bar T,\Omega}$.

For almost every $t \in [0,T_{\tau}(u_{\tau,k})]$, set $\tilde h_k(t) := \tilde f(t,y(t),y(t-\tau^1),u_{\tau,k}(t),u_{\tau,k}(t-\tau^2))$ and, if $T_{\tau}(u_{\tau,k}) + \tau^2 < \bar T$, extend it by 0 on $ (T_{\tau}(u_{\tau,k}),b]$. Now, we introduce several structures to deal with the presence of the control delay $\tau^2$. First, let
$$
\beta := \max \{ |f^0(t,x,y,u,v)| : 0 \le t \le b , \| (x,y) \| \le b , (u,v) \in \Omega^2 \} > 0
$$
Let $N \in \mathbb{N}$ such that $N \tau^2 \le \bar T < (N+1) \tau^2$. Taking continuous extensions, we clearly see that $y(\cdot)$ is defined on $[-\Delta,(N+1) \tau^2]$. Then, a.e. in $[0,\tau^2]$, we define
\begingroup
\tiny
\begin{equation} \label{redGuinn}
G(t,u^1,\dots,u^{N+1},\gamma^1,\dots,\gamma^{N+1}) = \left( \begin{array}{c}
f(t,y(t),y(t-\tau^1),u^1,\phi^2(t-\tau^2)) \\
f^0(t,y(t),y(t-\tau^1),u^1,\phi^2(t-\tau^2)) + \gamma^1 \\
f(t+\tau^2,y(t+\tau^2),y(t+\tau^2-\tau^1),u^2,u^1) \\
f^0(t+\tau^2,y(t+\tau^2),y(t+\tau^2-\tau^1),u^2,u^1) + \gamma^2 \\
\dots \\
f(t+N\tau^2,y(t+N\tau^2),y(t+N\tau^2-\tau^1),u^{N+1},u^N) \\
f^0(t+N\tau^2,y(t+N\tau^2),y(t+N\tau^2-\tau^1),u^{N+1},u^N) + \gamma^{N+1}
\end{array} \right)
\end{equation}
\endgroup
and
\begingroup
\footnotesize
\begin{eqnarray*}
\tilde V_{\beta}(t) := \bigg\{ &G(t,u^1,\dots,u^{N+1},\gamma^1,\dots,\gamma^{N+1}) : (u^1,\dots,u^{N+1}) \in \Omega^{N+1} , \forall i=1, \ldots, N+1 : \gamma^i \geq 0, \medskip \\
&|f^0(t,y(t),y(t-\tau^1),u^1,\phi^2(t-\tau^2)) + \gamma^1| \leq \beta,  \medskip \\
&\forall i=1, \ldots, N : |f^0(t+i\tau^2,y(t+i\tau^2),y(t+i\tau^2-\tau^1),u^{i+1},u^{i}) + \gamma^{i+1}| \leq \beta \bigg\}
\end{eqnarray*}
\endgroup
By Assumption $(A_1)$, $\tilde V_{\beta}(t)$ is compact for the standard topology of $\mathbb{R}^{(n+1)(N+1)}$. Moreover, it is not difficult to see that the assumption of affine dynamics and cost and $(A_1)$ ensure that $\tilde V_{\beta}(t)$ is convex for the same topology. Finally, we introduce
$$
\tilde{\mathcal{V}} := \{ \tilde g(\cdot) \in L^2([0,\tau^2],\mathbb{R}^{(n+1)(N+1)}) : \tilde g(t) \in \tilde V_{\beta}(t) , \textnormal{ a.e. } [0,\tau^2] \}
$$
Now, for every $i=0,\dots,N$, we denote
$$
\tilde g^{i+1}_k(t) := \tilde f(t+i\tau^2,y(t+i\tau^2),y(t+i\tau^2-\tau^1),u_{\tau,k}(t+i\tau^2),u_{\tau,k}(t+(i-1)\tau^2))
$$
and $\tilde g_k(t) = (\tilde g^{1}_k(t),\dots,\tilde g^{N+1}_k(t))$. Then $\tilde g_k(\cdot) \in \tilde{\mathcal{V}}$, for every $k \in \mathbb{N}$ and $\tilde{\mathcal{V}}$ is convex and closed in $L^2([0,\tau^2],\mathbb{R}^{(n+1)(N+1)})$ for the strong topology of $L^2$. From the second statement, it follows that $\tilde{\mathcal{V}}$ is convex and closed in $L^2([0,\tau^2],\mathbb{R}^{(n+1)(N+1)})$ for the weak topology of $L^2$. Since $(\tilde g_k(\cdot))_{k \in \mathbb{N}}$ is bounded in $L^2([0,\tau^2],\mathbb{R}^{(n+1)(N+1)})$, up to some subsequence, it converges for the weak topology of $L^2$ to a function $\tilde g(\cdot)$ that necessarily belongs to $\tilde{\mathcal{V}}$. We obtain that, for almost every $t \in [0,\tau^2]$ and $i=1,\dots,N+1$, there exist $\bar u^i_{\tau}(t) \in \Omega$ and $\bar \gamma^i_{\tau}(t) \geq 0$ such that
\begingroup
$$
\tilde g^{1}(t) := \left( \begin{array}{c}
f(t,y(t),y(t-\tau^1),\bar u^{1}_{\tau}(t),\phi^2(t-\tau^2)) \\
f^0(t,y(t),y(t-\tau^1),\bar u^{1}_{\tau}(t),\phi^2(t-\tau^2)) + \bar \gamma^1_{\tau}(t)
\end{array} \right)
$$
\endgroup
and, for every $i=1,\dots,N$,
\begingroup
$$
\tilde g^{i+1}(t) := \left( \begin{array}{c}
f(t+i\tau^2,y(t+i\tau^2),y(t+i\tau^2-\tau^1),\bar u^{i+1}_{\tau}(t),\bar u^{i}_{\tau}(t)) \\
f^0(t+i\tau^2,y(t+i\tau^2),y(t+i\tau^2-\tau^1),\bar u^{i+1}_{\tau}(t),\bar u^{i}_{\tau}(t)) + \bar \gamma^{i+1}_{\tau}(t)
\end{array} \right)
$$
\endgroup
Moreover, since $\Omega$ is compact, functions $\bar u^i_{\tau}(\cdot)$, $\bar \gamma^i_{\tau}(\cdot)$ can be chosen to be measurable on $[0,\tau^2]$ using a measurable selection lemma (see, e.g. \cite[Lemma 3A, page 161]{lee1967foundations}). Now, set
\begingroup
$$
\bar u_{\tau}(t) := \bigg\{ \begin{array}{cc}
\phi^2(t) & t \in [-\tau^2,0], \\
\bar u^i_{\tau}(t-i\tau^2) & t \in [i\tau^2,(i+1)\tau^2], i=0,\dots,N
\end{array}
$$
$$
\bar \gamma_{\tau}(t) := \bar \gamma^i_{\tau}(t-i\tau^2) \quad t \in [i\tau^2,(i+1)\tau^2], i=0,\dots,N
$$
\endgroup
which are measurable functions and let
$$
\tilde h(t) := \left( \begin{array}{c}
f(t,y(t),y(t-\tau^1),\bar u_{\tau}(t),\bar u_{\tau}(t-\tau^2)) \\
f^0(t,y(t),y(t-\tau^1),\bar u_{\tau}(t),\bar u_{\tau}(t-\tau^2)) + \bar \gamma_{\tau}(t)
\end{array} \right)
$$
From the weak convergence in $L^{\infty}$ of $(\tilde g_k(\cdot))_{k \in \mathbb{N}}$ towards $\tilde g(\cdot)$ it follows that $(\tilde h_k(\cdot))_{k \in \mathbb{N}}$ converges to $\tilde h(\cdot)$ for the weak topology of $L^2$. Since $\underset{k \rightarrow \infty}{\lim} \int_{0}^{\bar T} \varphi(t) \Big( \tilde F_k(t) - \tilde h_k(t) \Big) \; dt = 0$ for every $\varphi(\cdot) \in L^2([0,\bar T],\mathbb{R}^{n+1})$, we infer that $\tilde h = \tilde F$ almost everywhere in $[0,\bar T]$. Then
$$
y(t) := \phi^1(t) \mathds{1}_{[-\Delta,0)}(t) + \mathds{1}_{[0,\bar T]}(t) \bigg[ \phi^1(0) + \int_{0}^{t} f(t,y(t),y(t-\tau^1),\bar u_{\tau}(t),\bar u_{\tau}(t-\tau^2)) \; ds \bigg]
$$

It remains to show that control $\bar u_{\tau}(\cdot)$ is optimal for (\textbf{OCP})$_{\tau}$. First of all, since $M$ is compact and $(x_{\tau,k}(T_{\tau}(u_{\tau,k})))_{k \in \mathbb{N}} \subseteq M$, from what we showed previously, we necessarily obtain $y(T) \in M$. Furthermore, from what we showed above and by definition of weak star convergence, we have $C_{T_{\tau}(u_{\tau,k})}(\tau,u_{\tau,k}(\cdot)) \rightarrow \alpha$ and $C_{T_{\tau}(u_{\tau,k})}(\tau,u_{\tau,k}(\cdot)) \rightarrow \int_{0}^{\bar T} \big( f^0(t,y(t),y(t-\tau^1),\bar u_{\tau}(t),\bar u_{\tau}(t-\tau^2)) + \bar \gamma_{\tau}(t) \big) \; dt$. Since $\bar \gamma_{\tau}(\cdot)$ takes only non negative values, it follows that
$$
\int_{0}^{\bar T} f^0(t,y(t),y(t-\tau^1),\bar u_{\tau}(t),\bar u_{\tau}(t-\tau^2)) \; dt \le \alpha \le C_{T_{\tau}(v_{\tau})}(\tau,v_{\tau}(\cdot))
$$
for every $v_{\tau}(\cdot) \in \mathcal{U}^{\tau}_{\Omega}$. Then, $\bar \gamma_{\tau}(\cdot)$ is necessarily zero and the conclusion follows.

Now, consider pure state delays i.e., problems (\textbf{OCP})$_{\tau}$ where $\tau = (\tau^1,0)$. It is clear that, if Assumption $(B_2)$ holds, one can proceed with the same procedure as above (which is nothing else but the usual Filippov's argument) to achieve the existence of optimal controls. Of course, Guinn's reduction (\ref{redGuinn}) is not needed.

\begin{remark} \label{remarkGuinn}
Guinn's reduction (\ref{redGuinn}) converts the dynamics with control delays into a new dynamics without delays but with a larger number of variables. It is clear, from the context, that the natural assumption to provide the existence of optimal controls for generic nonlinear dynamics is the convexity of system (\ref{redGuinn}) for every $N \in \mathbb{N}$ (since the delay varies), which is a too strong assumption. From this, we see that the proof of Lemma 2.1 in \cite{nababan1979filippov} does not work under the weaker assumption of convexity of the epigraph of the extended dynamics.
\end{remark}
\subsubsection{Convergence of Trajectories and Controls}
Let us now establish the convergence properties of the trajectories and related controls and/or dynamics.

We start by considering general nonlinear dynamics and costs, assuming the existence of optimal controls of (\textbf{OCP})$_{\tau}$ for every $\tau = (\tau^1,\tau^2) \in (0,\tau_0)^2$. Assumptions $(C_1)$, $(C_2)$ and $(C_3)$ are now required. Note that this concerns also the case of control systems with dynamics and cost affine w.r.t. the two control variables satisfying $(C_1)$, whose existence of optimal controls was proved in the previous paragraph.

Let $(\tau_k)_{k \in \mathbb{N}} = ((\tau^1_k,\tau^2_k))_{k \in \mathbb{N}} \in (0,\tau_0)^2$ an arbitrary sequence converging to 0 as $k$ tends to $\infty$ and let $(x_{\tau_k}(\cdot),u_{\tau_k}(\cdot))$ be an optimal solution of (\textbf{OCP})$_{\tau_k}$ with final time $T_{\tau_k}(u_{\tau_k})$. Since $T_{\tau}(u_{\tau}) \in [0,b]$, up to some subsequence, the sequence $(T_{\tau_k})_{k \in \mathbb{N}} := (T_{\tau_k}(u_{\tau_k}))_{k \in \mathbb{N}}$ converges to some $\bar T \in [0,b]$ and, since $M$ is compact, the sequence
$(x_{\tau_k}(T_{\tau_k})_{k \in \mathbb{N}} \subseteq M$ converges up to some subsequence to a point in $M$.

For every integer $k$ and almost every $t \in [0,T_{\tau_k}]$, set $\tilde g_k(t) := \tilde f(t,x_{\tau_k}(t),x_{\tau_k}(t-\tau^1_k),u_{\tau_k}(t),u_{\tau_k}(t-\tau^2_k))$. By Assumption $(A_4)$, we extend $\tilde g_k(\cdot)$ by zero on $(T_{\tau_k},b]$. Assumptions $(A_1)$ and $(A_4)$ imply that the sequence $(\tilde g_k(\cdot))_{k \in \mathbb{N}}$ is bounded in $L^{\infty}$, then, up to some subsequence, it converges to some $\tilde g(\cdot) = (g(\cdot),g^0(\cdot)) \in L^{\infty}([0,b],\mathbb{R}^{n+1})$ for the weak star topology of $L^{\infty}$. Exploiting the weak star convergence of $L^{\infty}$ (using $\mathds{1}_{[\bar T,b]} \tilde g$ as test function), we get that $\tilde g(t) = 0$ for almost every $t \in [\bar T,b]$. Then, for every $t \in [0,\bar T]$, denote
\begin{equation} \label{trajZero}
\bar x(t) := \phi^1(t) \mathds{1}_{[-\Delta,0)}(t) + \mathds{1}_{[0,\bar T]}(t) \bigg[ \phi^1(0) + \int_{0}^{t} g(s) \; ds \bigg]
\end{equation}
Clearly, $\bar x(\cdot)$ is absolutely continuous and $\bar x(t) = \lim_{k \rightarrow \infty} x_{\tau_k}(t)$ pointwise in $[-\Delta,\bar T]$. Moreover, by $(A_1)$, $(A_4)$ and the Arzel\`{a}-Ascoli theorem, up to some subsequence, $x_{\tau_k}(\cdot)$ converges to $\bar x(\cdot)$, uniformly in $[-\Delta,\bar T]$. In particular, by continuity, $\bar x(\bar T) \in M$.

Let us prove that there exists a control $\bar u(\cdot) \in L^{\infty}([0,\bar T],\Omega)$ such that $\bar x(\cdot)$ is an admissible trajectory of (\textbf{OCP}) associated with control $\bar u(\cdot)$.

Using the definition of $\beta$ given previously, for every $t \in [0,\bar T]$, we consider the set
\begingroup
\small
$$
\tilde Z_{\beta}(t) := \bigg\{ \left( \begin{array}{c} f(t,\bar x(t),\bar x(t),u,v) \\
f^0(t,\bar x(t),\bar x(t),u,v) + \gamma
\end{array} \right) : (u,v) \in \Omega^2, \gamma \ge 0, | f^0(t,\bar x(t),\bar x(t),u,v) + \gamma | \le \beta \bigg\}
$$
\endgroup
Assumption $(C_3)$ provides that $\tilde Z_{\beta}(t)$ is compact and convex for the standard topology of $\mathbb{R}^{n+1}$. Then, let us consider the following statements.
\begin{itemize}
\item For every $\delta > 0$, $t \in [0,\bar T]$, the set $\tilde Z_{\delta}(t) := \{ x \in \mathbb{R}^{n+1} : d(x,\tilde Z_{\beta}(t)) \le \delta \}$, where $d$ is the standard set distance, is convex and compact for the standard topology of $\mathbb{R}^{n+1}$. This fact is straightforward.
\item For every $\delta > 0$, the set
$$
\tilde{\mathcal{Z}}_{\delta} := \{ \tilde  g(\cdot) \in L^2([0,\bar T],\mathbb{R}^{n+1}) : \tilde  g(t) \in \tilde  Z_{\delta}(t) \ \textnormal{for almost every} \ t \in [0,\bar T] \}
$$
is convex and closed in $L^2([0,\bar T],\mathbb{R}^{n+1})$ for the strong topology of $L^2$. Then, it is closed in $L^2([0,\bar T],\mathbb{R}^{n+1})$ for the weak topology of $L^2$.

Convexity is obvious. Let $(\tilde w_k(\cdot))_{k \in \mathbb{N}} \in \tilde{\mathcal{Z}}_{\delta}$ such that $\tilde w_k(\cdot) \xrightarrow{L^2} \tilde w(\cdot)$. Then $\tilde w(\cdot) \in L^2([0,\bar T],\mathbb{R}^{n+1})$ and there exists a subsequence such that $\tilde w_{k_m}(\cdot) \xrightarrow{a.e} \tilde w(\cdot)$. Since $\tilde Z_{\delta}(t)$ is closed for the standard topology of $\mathbb{R}^{n+1}$, for almost every $t \in [0, \bar T]$, it holds $\displaystyle \tilde w(t) = \lim_{m \rightarrow \infty} \tilde w_{k_m}(t) \in \tilde Z_{\delta}(t)$ and the statement follows.
\item For every $\delta > 0$, it exists an integer $k_{\delta}$ such that: $\forall k \geq k_{\delta} : \tilde g_k(\cdot) \in \tilde{\mathcal{Z}}_{\delta}$.

Indeed, thanks to assumptions $(A_1)$, $(A_4)$, mappings $f$, $f^0$ are globally Lipschitz onto $[0,\bar T] \times \overline{B^{2 n}_b(0)} \times \Omega^2$ and, using the mean value theorem, for almost every $t \in [0,\bar T]$
$$
\displaystyle \inf_{z \in \tilde Z_{\beta}(t)} \| \tilde g_k(t) - z \| \le C \Big[ \| x_{\tau_k}(t) - \bar x(t) \| + \| x_{\tau_k}(t - \tau^1_k) - \bar x(t) \| \Big]
$$
for a suitable constant $C > 0$. The conclusion follows.
\end{itemize}
Using the closedness of $\tilde{\mathcal{Z}}_{\delta}$ w.r.t. the weak topology of $L^2$, we infer that $\tilde g(\cdot) \in \tilde{\mathcal{Z}}_{\delta}, \ \forall \; \delta > 0$. This implies that $\tilde g(\cdot) \in \tilde{\mathcal{Z}}_0$. We obtain that, for almost every $t \in [0,\bar T]$, there exist $\bar v(t), \bar w(t) \in \Omega$ and $\bar \chi(t) \geq 0$ such that
\begingroup
\begin{equation} \label{gConvergence}
\tilde g(t) := \left( \begin{array}{c}
f(t,\bar x(t),\bar x(t),\bar v(t),\bar w(t)) \\
f^0(t,\bar x(t),\bar x(t),\bar v(t),\bar w(t)) + \bar \chi(t)
\end{array} \right)
\end{equation}
\endgroup
Moreover, since $\Omega$ is compact, functions $\bar v(\cdot)$, $\bar w(\cdot)$ and $\bar \chi(\cdot)$ can be chosen to be measurable on $[0,\bar T]$ using the usual standard measurable selection lemma cited previously. As final step, we want to show that $\int_{0}^{\bar T} \tilde g(t) \cdot \varphi(t) \; dt = \int_{0}^{\bar T} \tilde h(t) \cdot \varphi(t) \; dt$ for every test function $\varphi(\cdot) \in L^2([0,\bar T],\mathbb{R}^{n+1})$, where
$$
\tilde h(t) = \left( \begin{array}{c}
f(t,\bar x(t),\bar x(t),\bar u(t),\bar u(t)) \\
f^0(t,\bar x(t),\bar x(t),\bar u(t),\bar u(t)) + \bar \gamma(t)
\end{array} \right)
$$
with $\bar u(\cdot) \in L^{\infty}([0,\bar T],\Omega)$ and $\bar \gamma(\cdot)$ is a non-negative measurable function. From this and (\ref{trajZero}), it will follow that $(\bar x(\cdot),\bar u(\cdot))$ is the admissible trajectory of (\textbf{OCP}) sought. For this, we define $\tilde h_k(t) := \tilde f(t,x_{\tau_k}(t),x_{\tau_k}(t-\tau^1_k),u_{\tau_k}(t),u_{\tau_k}(t))$ which, with the obvious extensions presented previously, is bounded in $L^{\infty}([0,b],\mathbb{R}^{n+1})$. Then, up to some subsequence, it converges to some $\tilde h(\cdot) \in L^{\infty}([0,b],\mathbb{R}^{n+1})$ for the weak star topology of $L^{\infty}$. Using exactly the previous argument where $\tilde Z_{\beta}(t)$ is replaced by
\begingroup
\small
$$
\tilde W_{\beta}(t) := \bigg\{ \left( \begin{array}{c} f(t,\bar x(t),\bar x(t),u,u) \\
f^0(t,\bar x(t),\bar x(t),u,u) + \gamma
\end{array} \right) : u \in \Omega, \gamma \ge 0, | f^0(t,\bar x(t),\bar x(t),u,u) + \gamma | \le \beta \bigg\}
$$
\endgroup
it follows that there exist (again using $(C_3)$), for almost every $t \in [0,\bar T]$, $\bar u(t) \in \Omega$ and $\bar \gamma(t) \ge 0$ such that
$$
\tilde h(t) = \left( \begin{array}{c}
f(t,\bar x(t),\bar x(t),\bar u(t),\bar u(t)) \\
f^0(t,\bar x(t),\bar x(t),\bar u(t),\bar u(t)) + \bar \gamma(t)
\end{array} \right)
$$
Moreover, as done previously, the functions $\bar u(\cdot)$ and $\bar \gamma(\cdot)$ can be chosen to be measurable on $[0,\bar T]$. The last statement follows by proving that, up to some subsequence,
$$
\lim_{k \rightarrow \infty} \int_{0}^{\bar T} \bigg( \tilde f(t,x_{\tau_k}(t),x_{\tau_k}(t-\tau^1_k),u_{\tau_k}(t),u_{\tau_k}(t-\tau^2_k))
$$
$$
\qquad \qquad \qquad \qquad \qquad - \tilde f(t,x_{\tau_k}(t),x_{\tau_k}(t-\tau^1_k),u_{\tau_k}(t),u_{\tau_k}(t)) \bigg) \cdot \varphi(t) \; dt = 0
$$
for every test function $\varphi(\cdot) \in L^2([0,\bar T],\mathbb{R}^{n+1})$. Then, fix $\varphi(\cdot) \in L^2([0,\bar T],\mathbb{R}^{n+1})$ and consider the \textit{shift operator}
$$
S_{\tau} : L^2(\mathbb{R},\mathbb{R}^m) \rightarrow L^2(\mathbb{R},\mathbb{R}^m) : \Big( t \mapsto \phi(t) \Big) \mapsto \Big( t \mapsto \phi(t - \tau) \Big)
$$
Using the dominated convergence theorem, it is clear that, for every $\phi(\cdot) \in L^2(\mathbb{R},\mathbb{R}^m)$, it holds $\| S_{\tau} \phi - \phi \|_{L^2} \rightarrow 0$ when $\tau \rightarrow 0$. Extending $u_{\tau_k}(\cdot)$, $u_{\tau_k}(\cdot-\tau^2_k)$ by zero out $[-\Delta,b]$ and considering Assumption $(A_4)$, up to some subsequence, $(u_{\tau_k}(\cdot))_{k \in \mathbb{N}}$ converges to some $\bar \mu(\cdot) \in L^2_{\textnormal{loc}}(\mathbb{R},\mathbb{R}^m)$ for the weak topology of $L^2_{\textnormal{loc}}$. Necessarily, up to some subsequence, $(u_{\tau_k}(\cdot-\tau_k))_{k \in \mathbb{N}}$ also converges to $\bar \mu(\cdot)$ for the weak topology of $L^2_{\textnormal{loc}}$. Indeed, for every $\eta(\cdot) \in L^2_{\textnormal{loc}}(\mathbb{R},\mathbb{R}^m)$, one has
$$
\int_{\mathbb{R}} (u_{\tau_k}(t-\tau_k) - \bar \mu(t)) \cdot \eta(t) \; dt = \int_{\mathbb{R}} (u_{\tau_k}(t) - \bar \mu(t)) \cdot \Big( S_{-\tau_k} \eta \Big)(t) \; dt + \int_{\mathbb{R}} (S_{\tau_k} \bar \mu - \bar \mu)(t) \cdot \eta(t) \; dt
$$
$$
= \int_{\mathbb{R}} (u_{\tau_k}(t) - \bar \mu(t)) \cdot \eta(t) \; dt + \int_{\mathbb{R}} (u_{\tau_k}(t) - \bar \mu(t)) \cdot \Big( S_{-\tau_k} \eta - \eta \Big)(t) \; dt + \int_{\mathbb{R}} (S_{\tau_k} \bar \mu - \bar \mu)(t) \cdot \eta(t) \; dt
$$
which clearly converges to 0. Now, Assumption $(C_1)$ gives that both $(u_{\tau_k}(\cdot))_{k \in \mathbb{N}}$ and $(u_{\tau_k}(\cdot-\tau_k))_{k \in \mathbb{N}}$ converge to $\bar \mu(\cdot)$ for the strong topology of $L^1_{\textnormal{loc}}$ (see \cite[Corollary 1]{visintin1984strong}). This implies that, up to some subsequence, $(u_{\tau_k}(\cdot) - u_{\tau_k}(\cdot-\tau^2_k))_{k \in \mathbb{N}}$ converges to 0 a.e. in $[0,\bar T]$. Then, thanks to Assumption $(C_2)$, up to some subsequence,
\begingroup
\scriptsize
$$
\left| \int_{0}^{\bar T} \bigg( \tilde f(t,x_{\tau_k}(t),x_{\tau_k}(t-\tau^1_k),u_{\tau_k}(t),u_{\tau_k}(t-\tau^2_k)) - \tilde f(t,x_{\tau_k}(t),x_{\tau_k}(t-\tau^1_k),u_{\tau_k}(t),u_{\tau_k}(t)) \bigg) \cdot \varphi(t) \; dt \right|
$$
\endgroup
$$
\le \Lambda \ \| \varphi(\cdot) \|_{L^2} \ \| u_{\tau_k}(\cdot-\tau^2_k) - u_{\tau_k}(\cdot) \|_{L^2} \rightarrow 0
$$
where $\Lambda \ge 0$ is a constant and we use the dominated convergence theorem.

We complete this step by showing that $\bar T = T$ and $(\bar x(\cdot),\bar u(\cdot)) = (x(\cdot),u(\cdot))$. We denote by $C_S({\tau},w(\cdot))$ the cost of (\textbf{OCP})$_{\tau}$ given by the admissible control $w(\cdot)$ and the final time $S$. First, the previous argument shows that $C_{T_{\tau_k}}(\tau_k,u_{\tau_k}(\cdot)) \rightarrow C_{\bar T}(0,\bar u(\cdot)) + \int_{0}^{\bar T} \bar \gamma(t) \; dt$. By Section \ref{controllability}, for every integer $k$, there exists a sequence $(S_k,w_k(\cdot),y_k(\cdot))$ respectively of final times, admissible controls and trajectories of (\textbf{OCP})$_{\tau_k}$ which converges to $(T,u(\cdot),x(\cdot))$ (with the evident topologies) as $k$ converges to $\infty$. By optimality, we have $C_{T_{\tau_k}}(\tau_k,u_{\tau_k}(\cdot)) \le C_{S_k}(\tau_k,w_k(\cdot))$. Since $\bar \gamma(t) \ge 0$, passing to the limit gives $C_{\bar T}(0,\bar u(\cdot)) \le C_T(u(\cdot))$ and from Assumption $(A_2)$ we infer that $\bar \gamma(\cdot) = 0$ and $\bar T = T$, $(\bar x(\cdot),\bar u(\cdot)) = (x(\cdot),u(\cdot))$.

All the previous arguments can be used to show that also, up to some subsequence, $(u_{\tau_k}(\cdot))_{k \in \mathbb{N}}$ converges to $u(\cdot)$ almost everywhere in $[0,T]$. Indeed, if $\bar \mu(\cdot) \in L^2_{\textnormal{loc}}(\mathbb{R},\mathbb{R}^m)$ is the previous limit for the weak topology of $L^2_{\textnormal{loc}}$ of a subsequence of $(u_{\tau_k}(\cdot))_{k \in \mathbb{N}}$, the previous result shows that $\int_{0}^{\bar T} \tilde f(t,\bar x(t),\bar x(t),\bar u(t),\bar u(t)) \cdot \varphi(t) \; dt = \int_{0}^{\bar T} \tilde f(t,\bar x(t),\bar x(t),\bar \mu(t),\bar \mu(t)) \cdot \varphi(t) \; dt$ for every test function $\varphi(\cdot) \in L^2([0,\bar T],\mathbb{R}^{n+1})$ which means that $u(\cdot) = \bar u(\cdot)$ generates the same trajectory and cost as $\mu(\cdot)$, and so, by uniqueness, $\mu(\cdot) = u(\cdot)$.

Now, we consider the case consisting of pure state delays i.e., problems (\textbf{OCP})$_{\tau}$ such that $\tau = (\tau^1,0) \in (0,\tau_0) \times \{ 0 \}$. It is not difficult to see that the argument above until expression (\ref{gConvergence}) can be iterated in this context, with the help of Assumption $(B_2)$, by substituting $\tilde Z_{\beta}(t)$ with
\begingroup
\small
$$
\bigg\{ \left(f_1(t,x,y,u),f^0_1(t,x,y,u)+\gamma,\frac{\partial \tilde f_1}{\partial x}(t,x,y,u),\frac{\partial \tilde f_1}{\partial y}(t,x,y,u)\right)
$$
$$
: u \in \Omega, \gamma \ge 0, | f^0_1(t,\bar x(t),\bar x(t),u,v) + \gamma | \le \beta \bigg\}
$$
\endgroup
to obtain the sought convergence of trajectories, costs and dynamics (instead of optimal controls). Moreover, the convexity of the epigraph concerning the derivatives of the dynamics leads to the following weak convergences (useful in the sequel)
\begingroup
\small
\begin{eqnarray} \label{derDynConv}
\begin{cases}
\displaystyle \frac{\partial \tilde f_1}{\partial x}(\cdot,x_{\tau_k}(\cdot),x_{\tau_k}(\cdot-\tau^1_k),u_{\tau_k}(\cdot)) \overset{(L^{\infty})^*}{\rightharpoonup} \frac{\partial \tilde f_1}{\partial x}(\cdot,x(\cdot),x(\cdot),u(\cdot)) \medskip \\
\displaystyle \frac{\partial \tilde f_1}{\partial y}(\cdot,x_{\tau_k}(\cdot),x_{\tau_k}(\cdot-\tau^1_k),u_{\tau_k}(\cdot)) \overset{(L^{\infty})^*}{\rightharpoonup} \frac{\partial \tilde f_1}{\partial y}(\cdot,x(\cdot),x(\cdot),u(\cdot))
\end{cases}
\end{eqnarray}
\endgroup

At this step, for every considered case, we have shown that $(T,x(\cdot),u(\cdot))$ (substituted by $(T,x(\cdot),\dot{x}(\cdot))$ for the case of pure state delays) is the unique closure point (for the topologies used above) of $(T_{\tau_k},x_{\tau_k}(\cdot),u_{\tau_k}(\cdot))_{k \in \mathbb{N}}$ ($(T_{\tau_k},x_{\tau_k}(\cdot),\dot{x}_{\tau_k}(\cdot))_{k \in \mathbb{N}}$ for the cases of pure state delays), for any (sub)sequence of delays $(\tau_k)_{k \in \mathbb{N}}$ converging to 0 and therefore the convergence holds as well for the whole family $(T_{\tau},x_{\tau}(\cdot),u_{\tau}(\cdot))_{\tau \in (0,\tau_0)^2}$ ($(T_{\tau},x_{\tau}(\cdot),\dot{x}_{\tau}(\cdot))_{\tau \in (0,\tau_0)^2}$ for the cases of pure state delays).

\subsubsection{Convergence of the Adjoint Vectors}

In the sequel, $(x_{\tau}(\cdot),u_{\tau}(\cdot))$ will denote an optimal solution of (\textbf{OCP})$_{\tau}$ defined on $[-\Delta,T_{\tau}]$ such that, if needed, it is extended continuously in $[-\Delta,T]$. From the PMP, it follows that $x_{\tau}(\cdot)$ is the projection of an extremal $(x_{\tau}(\cdot),p_{\tau}(\cdot),p^0_{\tau},u_{\tau}(\cdot))$ which satisfies (\ref{dynDual}). From now on, we consider that either Assumption $(B_1)$ or Assumption $(C_1)$ are satisfied, depending on whether we consider pure state delays or not. We remind that $(C_1)$ implies the convergence almost everywhere of $(u_{\tau}(\cdot))_{\tau \in (0,\tau_0)^2}$ towards $u(\cdot)$.

The main step of this part consists in showing the convergence of the Pontryagin cone of (\textbf{OCP})$_{\tau}$ towards the Pontryagin cone related to (\textbf{OCP}). Since the definition of variation vectors relies on Lebesgue points of optimal controls, we need first an appropriate set of converging Lebesgue points.
\begin{lemma} \label{lemma1}
Consider (\textbf{OCP})$_{\tau}$ with pure state delays i.e., $\tau = (\tau^1,0)$. For every $s \in (0,T)$, Lebesgue point of function $\tilde f_1(\cdot,x(\cdot),x(\cdot),u(\cdot),u(\cdot))$, there exists a family $(s_{\tau})_{\tau^1 \in (0,\tau_0)} \subseteq [s,T)$, which are Lebesgue points of $\tilde f_1(\cdot,x_{\tau}(\cdot),x_{\tau}(\cdot-\tau^1),u_{\tau}(\cdot))$, such that $\tilde f_1(s_{\tau},x_{\tau}(s_{\tau}),x_{\tau}(s_{\tau}-\tau^1),u_{\tau}(s_{\tau})) \rightarrow \tilde f_1(s,x(s),x(s),u(s))$ and $s_{\tau} \rightarrow s$ as $\tau \rightarrow 0$.

Conversely, if we consider (\textbf{OCP})$_{\tau}$ with general delays $\tau = (\tau^1,\tau^2)$, for every $s \in (0,T)$, Lebesgue point of $u(\cdot)$, there exists a family $(s_{\tau})_{\tau \in (0,\tau_0)^2} \subseteq [s,T)$, which are Lebesgue points of $u_{\tau}(\cdot)$, $u_{\tau}(\cdot-\tau^2)$ and of $u_{\tau}(\cdot+\tau^2)$, such that $u_{\tau}(s_{\tau}) \rightarrow u(s)$, $u_{\tau}(s_{\tau}-\tau^2) \rightarrow u(s)$, $u_{\tau}(s_{\tau}+\tau^2) \rightarrow u(s)$ and $s_{\tau} \rightarrow s$ as $\tau \rightarrow 0$.
\end{lemma}
\begin{proof}[Proof of Lemma \ref{lemma1}] We start by proving the first assertion. Denote $h^{\tau}(t) = (h^{\tau}_1(t),\dots,h^{\tau}_{n+1}(t)) := \tilde f_1(t,x_{\tau}(t),x_{\tau}(t-\tau^1),u_{\tau}(t))$ and $h(t) = (h_1(t),\dots,h_{n+1}(t)) := \tilde f_1(t,x(t),x(t),u(t))$. Let us prove that, for every $s \in (0,T)$ Lebesgue point of function $h(\cdot)$, for every $\beta > 0$, $\alpha_s > 0$ (small enough so that $s+\alpha_s<T$), there exists $\gamma_{s,\alpha_s,\beta} > 0$ such that, for every $\tau^1 \in (0,\gamma_{s,\alpha_s,\beta})$, there exists $s_{\tau} \in [s,s+\alpha_s]$ Lebesgue point of function $h^{\tau}(\cdot)$ for which $\| h^{\tau}(s_{\tau}) - h(s) \| < \beta$. By contradiction, suppose that there exists $s \in (0,T)$ Lebesgue point of the function $h(\cdot)$, $\alpha_{s} > 0$, $\beta > 0$ such that, for every integer $k$, there exists $\tau_k \in (0,1/k) \times \{ 0 \}$ and $i_k \in \{ 1,\dots,n+1 \}$ for which, for every $t \in [s,s+\alpha_{s}]$ Lebesgue point of function $h^{\tau_k}(\cdot)$, there holds $|h^{\tau_k}_{i_k}(t) - h_{i_k}(s)| \geq \beta$. From the previous results, the family $(h^{\tau}(\cdot))_{\tau \in (0,\tau_0) \times \{ 0 \}}$ converges to $h(\cdot)$ in $L^{\infty}$ for the weak star topology. From this, the restriction of the components of $(h^{\tau}(\cdot))_{\tau \in (0,\tau_0) \times \{ 0 \}}$ to $[s,s+\alpha_{s}]$ converges to the restriction of the components of $h(\cdot)$ to $[s,s+\alpha_{s}]$ in the $L^1$ norm. This means that, for every $\varepsilon > 0$ and component $i$, there exists an integer $k_{\varepsilon,i}$ such that, for every $k \geq k_{\varepsilon,i}$, it holds $\Big| \int_s^{s+\alpha_s} h^{\tau_k}_i(t) \ dt - \int_s^{s+\alpha_s} h_i(t) \ dt \Big| < \varepsilon$. Since $s$ is a Lebesgue point of $h(\cdot)$, it exists a real $\delta_{\alpha_s,i}$ such that $\int_s^{s+\alpha_s} h_i(s) \ ds = h_i(s) + \delta_{\alpha_s,i}$. Without loss of generality, we can assume $\beta > \delta_{\alpha_s,i}$ and let $M > 0$ such that $\beta/M + \delta_{\alpha_s,i} < \beta$, $-\beta/M + \delta_{\alpha_s,i} > -\beta$. Assumption $(B_1)$ implies that $h^{\tau}(\cdot)$ is continuous for every $\tau^1 > 0$, and then, there exists $t_{k,i} \in [s,s+\alpha_s]$ and an integer $k_{s,\alpha_s,\beta,i}$ such that, for every $i$ and every $k \geq k_{s,\alpha_s,\beta,i}$, $\Big| h^{\tau_k}_i(t_{k,i}) - \int_s^{s+\alpha_s} h_i(t) \ dt \Big| < \beta \alpha_s/M $. This implies that there exists an integer $\tilde k := \max_i k_{s,\alpha_s,\beta,i}$ such that for every $\tau_k \in (0,1/\tilde k) \times \{ 0 \}$ and every $i \in \{ 1,\dots,n+1 \}$ there exists $t_{k,i} \in [s,s+\alpha_s]$ Lebesgue point of function $h^{\tau_k}(\cdot)$ (because $h^{\tau_k}(\cdot)$ is continuous) such that $|h^{\tau_k}_i(t_{k,i}) - h_i(s)| < \beta$, a contradiction.

Now, we consider the second statement. Without loss of generality, we extend $u_{\tau}(\cdot)$ by some constant vector of $\Omega$ in $[T_{\tau},b]$. Denote $h_{\tau}(t) = (h^1_{\tau}(t),\dots,h^{3m}_{\tau}(t)) := \Big( u_{\tau}(t),u_{\tau}(t-\tau^2),u_{\tau}(t+\tau^2) \Big)$, $h(t) = (h^1(t),\dots,h^{3m}(t)) := \Big( u(t),u(t),u(t) \Big)$ and fix $s \in (0,T)$, Lebesgue point of $h(\cdot)$. By contradiction, suppose that there exist $\beta > 0$ and $\alpha > 0$ such that, for every integer $k$, there exist $\tau_k = (\tau^1_k,\tau^2_k) \in (0,1/k)^2$ and $i_k \in \{ 1,\dots,3m \}$ for which, for every $r \in [s,s+\alpha]$ Lebesgue point of $h_{\tau_k}(\cdot)$, it holds $| h^{i_k}_{\tau_k}(r) - h^{i_k}(s) | \ge \beta$. From the previous argument, up to some extension, the family $(u_{\tau}(\cdot))_{\tau \in (0,\tau_0)^2}$ converges to $u(\cdot)$ almost everywhere in $[0,T]$. Clearly, the same holds true for $(u_{\tau}(\cdot-\tau^2))_{\tau \in (0,\tau_0)^2}$ and $(u_{\tau}(\cdot+\tau^2))_{\tau \in (0,\tau_0)^2}$. Then, $(h^{i}_{\tau_k}(\cdot))_{k \in \mathbb{N}}$ converges almost everywhere to $h^i(\cdot)$, raising a contradiction.
\end{proof}

We are now able to prove the convergence property for Pontryagin cones.
\begin{lemma} \label{prop1}
For every $\tilde v \in \tilde K^0(T)$ and every $\tau = (\tau^1,\tau^2) \in (0,\tau_0 )^2$ (or $\tau = (\tau^1,0) \in (0,\tau_0) \times \{ 0 \}$ in the case of pure state delays), there exists $\tilde v_{\tau} \in \tilde K^{\tau}(T_{\tau})$ such that $\tilde v_{\tau}$ converges to $\tilde v$ as $\tau$ tends to 0.
\end{lemma}
\begin{proof}[Proof of Lemma \ref{prop1}] We prove the statement for problems (\textbf{OCP})$_{\tau}$ with general delays $\tau = (\tau^1,\tau^2)$. If pure state delays $\tau = (\tau^1,0)$ are considered, the same guideline can be used by exploiting Lemma \ref{lemma1} and (\ref{derDynConv}).

Suppose first that $\tilde v = \tilde w^0_{s,z}(T)$, where $z \in \Omega$ and $0<s<T$ is a Lebesgue point of $u(\cdot)$. By definition, $\tilde w^0_{s,z}(\cdot)$ is the solution of
\begin{eqnarray} \label{sysProp1}
\begin{cases}
\dot{\psi}(t) = \displaystyle \bigg( \frac{\partial \tilde f}{\partial x}(t,x(t),x(t),u(t),u(t)) + \frac{\partial \tilde f}{\partial y}(t,x(t),x(t),u(t),u(t)) \bigg) \psi(t) \medskip \\
\psi(s) = \tilde f(s,x(s),x(s),z,z) - \tilde f(s,x(s),x(s),u(s),u(s))
\end{cases}
\end{eqnarray}

From Lemma \ref{lemma1}, there exists a family $(s_{\tau})_{\tau \in (0,\tau_0)^2} \subseteq [s,T)$, which are Lebesgue points of $u_{\tau}(\cdot)$, $u_{\tau}(\cdot-\tau^2)$ and of $u_{\tau}(\cdot+\tau^2)$, such that $u_{\tau}(s_{\tau}) \rightarrow u(s)$, $u_{\tau}(s_{\tau}-\tau^2) \rightarrow u(s)$, $u_{\tau}(s_{\tau}+\tau^2) \rightarrow u(s)$ and $s_{\tau} \rightarrow s$ as $\tau \rightarrow 0$. Then, we are able to consider $\tilde v^{\tau}_{s_{\tau},\omega^-_z(s_{\tau})}(\cdot)$ and $\tilde v^{\tau}_{s_{\tau}+\tau^2,\omega^+_z(s_{\tau})}(\cdot)$, solutions of (\ref{dynVariation}) with initial data provided respectively by (\ref{omegaMinus}) and (\ref{omegaPlus}), and we denote $\tilde w^{\tau}_{s_{\tau},z}(t) = \tilde v^{\tau}_{s_{\tau},\omega^-_z(s_{\tau})}(t) + \tilde v^{\tau}_{s_{\tau}+\tau^2,\omega^+_z(s_{\tau})}(t)$. Using the particular form of the dynamics and the cost of (\textbf{OCP})$_{\tau}$, Lemma \ref{lemma1} gives $\lim_{\tau \rightarrow 0} \Big( \omega^-_z(s_{\tau}) + \omega^+_z(s_{\tau}) \Big) = \tilde f(s,x(s),x(s),z,z) - \tilde f(s,x(s),x(s),u(s),u(s))$. Moreover, from the previous sections, $\frac{\partial \tilde f}{\partial x}(\cdot,x_{\tau}(\cdot),x_{\tau}(\cdot-\tau^1),u_{\tau}(\cdot),u_{\tau}(\cdot-\tau^2))$ and $\frac{\partial \tilde f}{\partial y}(\cdot,x_{\tau}(\cdot),x_{\tau}(\cdot-\tau^1),u_{\tau}(\cdot),u_{\tau}(\cdot-\tau^2))$ tend respectively to $\frac{\partial \tilde f}{\partial x}(\cdot,x(\cdot),x(\cdot),u(\cdot),u(\cdot))$ and $\frac{\partial \tilde f}{\partial y}(\cdot,x(\cdot),x(\cdot),u(\cdot),u(\cdot))$ for the weak star topology of $L^{\infty}$, as $\tau \rightarrow 0$. By the continuous dependence w.r.t. intial data and since $T_{\tau}$ converges to $T$, it follows immediately that $\tilde v_{\tau} := \tilde w^{\tau}_{s_{\tau},z}(T_{\tau}) \in \tilde K^{\tau}(T_{\tau})$ converges to $\tilde v$ as $\tau$ tends to 0. The extension to general $\tilde v \in \tilde K^0(T)$ is then straightforward.
\end{proof}

For the last part of the proof, an iterative use of Lemma \ref{prop1} is done. In this context, a distinction between problems with pure state delays and problems with general delays is not needed anymore because the concerned proofs are the same. Then, we focus only on (\textbf{OCP})$_{\tau}$ with general delays $\tau = (\tau^1,\tau^2)$. Assumptions $(B_1)$ and $(C_1)$ are always implicitly assumed. \\

We need first to prove that, considering if necessary a smaller $\tau_0 > 0$, for every $\tau = (\tau^1,\tau^2) \in (0,\tau_0)^2$, every extremal lift $(x_{\tau}(\cdot),p_{\tau}(\cdot),p^0_{\tau},u_{\tau}(\cdot))$ of any solution of (\textbf{OCP})$_{\tau}$ is normal. Moreover, the set $\{ p_{\tau}(T_{\tau}) : \tau \in (0,\tau_0)^2 \}$ is bounded.

We consider the first statement proceeding by contradiction. Assume that, for every integer $k$, there exist $\tau_k = (\tau^1_k,\tau^2_k) \in (0,1/k)^2$ and a solution $(x_{\tau_k}(\cdot),u_{\tau_k}(\cdot))$ of (\textbf{OCP})$_{\tau_k}$ having an abnormal extremal lift $(x_{\tau_k}(\cdot),p_{\tau_k}(\cdot),0,u_{\tau_k}(\cdot))$. Set $\psi_{\tau_k} = p_{\tau_k}(T_{\tau_k})$ for every integer $k$. Then, there holds $\Big\langle (\psi_{\tau_k},0) , \tilde v_{\tau_k} \Big\rangle \leq 0$, for every $\tilde v_{\tau_k} \in \tilde K^{\tau_k}(T_{\tau_k})$, and, since the final time is free, we infer that $\langle \psi_{\tau_k} , f(T_{\tau_k},x_{\tau_k}(T_{\tau_k}),x_{\tau_k}(T_{\tau_k} - \tau^1_k),u_{\tau_k}(T_{\tau_k}),u_{\tau_k}(T_{\tau_k} - \tau^2_k)) \rangle = 0$, for every integer $k$. Since the final adjoint vector $(p_{\tau_k}(\cdot),p^0_{\tau_k}=0)$ is defined up to a multiplicative scalar, we assume that $\psi_{\tau_k}$ is a unit vector for every integer $k$. Then, up to a subsequence, the sequence $(\psi_{\tau_k})_{k \in \mathbb{N}}$ converges to some unit vector $\psi \in \mathbb{R}^n$. Using the previous results, passing to the limit we infer that $\Big\langle (\psi,0) , \tilde v \Big\rangle \leq 0$ for every $\tilde v \in \tilde K^0(T)$ and that $\langle \psi , f(T,x(T),x(T),u(T),u(T)) \rangle = 0$. It then follows that $(x(\cdot),u(\cdot))$ has an abnormal extremal lift. This contradicts Assumption $(A_3)$.

For the second statement, again by contradiction, assume that there exists a sequence $(\tau_k = (\tau^1_k,\tau^2_k))_{k \in \mathbb{N}} \subseteq (0,\tau_0)^2$ converging to 0 such that $\| p_{\tau_k}(T_{\tau_k}) \|$ tends to $+\infty$. Since the sequence $\left( \frac{p_{\tau_k}(T_{\tau_k})}{\| p_{\tau_k}(T_{\tau_k}) \|} \right)_{k \in \mathbb{N}}$ is bounded in $\mathbb{R}^n$, up to some subsequence, it converges to some unit vector $\psi$. By the construction of the adjoint vector, we have $\Big\langle (p_{\tau_k}(T_{\tau_k}),-1) , \tilde v_{\tau_k} \Big\rangle \leq 0$, for every $\tilde v_{\tau_k} \in \tilde K^{\tau_k}(T_{\tau_k})$, and
\begingroup
$$
\langle p_{\tau_k}(T_{\tau_k}) , f(T_{\tau_k},x_{\tau_k}(T_{\tau_k}),x_{\tau_k}(T_{\tau_k} - \tau^1_k),u_{\tau_k}(T_{\tau_k}),u_{\tau_k}(T_{\tau_k} - \tau^2_k)) \rangle
$$
$$
- f^0(T_{\tau_k},x_{\tau_k}(T_{\tau_k}),x_{\tau_k}(T_{\tau_k} - \tau^1_k),u_{\tau_k}(T_{\tau_k}),u_{\tau_k}(T_{\tau_k} - \tau^2_k)) = 0
$$
\endgroup
for every integer $k$. Dividing by $\| p_{\tau_k}(T_{\tau_k}) \|$ and passing to the limit, thanks to the previous results, it follows that the solution $(x(\cdot),u(\cdot))$ has an abnormal extremal lift, which again contradicts Assumption $(A_3)$. \\

Now, let $\psi$ be a closure point of $\{ p_{\tau}(T_{\tau}) : \tau \in (0,\tau_0)^2 \}$ and $(\tau_k = (\tau^1_k,\tau^2_k))_{k \in \mathbb{N}} \subseteq (0,\tau_0)^2$ a sequence converging to 0 such that $p_{\tau_k}(T_{\tau_k})$ tends to $\psi$.  Using the continuous dependence w.r.t. initial data and the convergence properties established, we infer that the sequence $(p_{\tau_k}(\cdot))_{k \in \mathbb{N}}$ converges uniformly to the solution $z(\cdot)$ of the Cauchy problem
\begingroup
\small
$$
\begin{cases}
\displaystyle \dot{z}(t) = -\frac{\partial H}{\partial x}(t,x(t),x(t),z(t),-1,u(t),u(t)) - \frac{\partial H}{\partial y}(t,x(t),x(t),z(t),-1,u(t),u(t)) \medskip \\
z(T) = \psi
\end{cases}
$$
\endgroup
Moreover, since $\Big\langle (p_{\tau_k}(T_{\tau_k}),-1) , \tilde v_{\tau_k} \Big\rangle \leq 0$, for every $\tilde v_{\tau_k} \in \tilde K^{\tau_k}(T_{\tau_k})$ and
\begingroup
$$
\langle p_{\tau_k}(T_{\tau_k}) , f(T_{\tau_k},x_{\tau_k}(T_{\tau_k}),x_{\tau_k}(T_{\tau_k} - \tau^1_k),u_{\tau_k}(T_{\tau_k}),u_{\tau_k}(T_{\tau_k} - \tau^2_k)) \rangle
$$
$$
- f^0(T_{\tau_k},x_{\tau_k}(T_{\tau_k}),x_{\tau_k}(T_{\tau_k} - \tau^1_k),u_{\tau_k}(T_{\tau_k}),u_{\tau_k}(T_{\tau_k} - \tau^2_k)) = 0
$$
\endgroup
for every integer $k$, passing to the limit, thanks to the previous results, we obtain $\Big\langle (\psi,-1) , \tilde v \Big\rangle \leq 0$, for every $\tilde v \in \tilde K^0(T)$ and $\langle \psi , f(T,x(T),x(T),u(T),u(T)) \rangle - f^0(T,x(T),x(T),u(T),u(T)) = 0$. It follows that $(x(\cdot),z(\cdot),-1,u(\cdot))$ is a normal extremal lift of (\textbf{OCP}). By $(A_3)$, we obtain $z(\cdot) = p(\cdot)$ in $[0,T]$, giving the conclusion.

\section{Proof of Lemma \ref{iFTP}} \label{appConic}

We start by recalling the following standard result (see \cite{agrachev2013control} for a proof).
\begin{lemma}
Let $\ell : \mathbb{R}^m \rightarrow \mathbb{R}^n$ be a linear mapping such that $\ell(\mathbb{R}^m_+) = \mathbb{R}^n$. Then,
\begin{itemize}
\item we have $m > n + 1$ and the intersection $(0,+\infty)^m \cap \textnormal{ker } \ell$ is nontrivial;
\item there exists a $n$-dimensional subspace $W \subseteq{\mathbb{R}^m}$ such that $\ell|_W : W \rightarrow \mathbb{R}^n$ is an isomorphism.
\end{itemize}
\end{lemma}
Applying this lemma to $\ell = \frac{\partial F}{\partial x}(0,0,0)$ yields the existence of a nontrivial vector $v \in (0, +\infty)^m$ such that $\ell(v) = 0$, and the existence of a $n$-dimensional subspace $W \subseteq \mathbb{R}^m$ such that $\ell|_W : W \rightarrow \mathbb{R}^n$ is an isomorphism.

Let $\delta > 0$ small enough such that $v + B_{\delta} = v + W \cap \bar B¯_{\delta}(0) \subseteq (0,+\infty)^m$. The set $U_{\delta} = \ell(B_{\delta})$ is then a closed neighborhood of 0 in $\mathbb{R}^n$. For every $\varepsilon_1$, $\varepsilon_2 \ge 0$, $y, u \in \mathbb{R}^n$, set $\Phi(\varepsilon_1,\varepsilon_2,y,u) = u - F(\varepsilon_1,\varepsilon_2,{\ell|_W}^{-1}(u)) + y$. There holds $\Phi(0,0,0,0) = 0$, and, for almost every $\varepsilon_1$, $\varepsilon_2 \ge 0$, $y, u_1, u_2 \in \mathbb{R}^n$, one has
\begingroup
\footnotesize
$$
\Phi(\varepsilon_1,\varepsilon_2,y,u_1) - \Phi(\varepsilon_1,\varepsilon_2,y,u_2) = u_1 - u_2 + \frac{\partial F}{\partial x}(\varepsilon_1,\varepsilon_2,0).{\ell|_W}^{-1} (u_2 - u_1) + \| u_2 - u_1 \| g(\varepsilon_1,\varepsilon_2,u_1, u_2)
$$
\endgroup
where $g(\varepsilon_1,\varepsilon_2,u_1,u_2) \rightarrow 0$ as soon as $\| u_1 \| + \| u_2 \| \xrightarrow{\text{a.e.}} 0$. Since $\frac{\partial F}{\partial x}$ is approximately continuous w.r.t. $(\varepsilon_1,\varepsilon_2)$, it follows that, up to sets of measure zero, the mapping $\frac{\partial F}{\partial x}(\varepsilon_1,\varepsilon_2,0).{\ell|_W}^{-1}$ is close to the identity for almost every $(\varepsilon_1,\varepsilon_2)$ small enough. Therefore, by continuity, there exist $k \in (0,1)$ and $\varepsilon_0 > 0$ such that, for every $(\varepsilon_1,\varepsilon_2) \in [0,\varepsilon_0)^2$, $y \in \mathbb{R}^n$, the mapping $u \mapsto \Phi(\varepsilon_1,\varepsilon_2,y,u)$ is $k$-Lipschitzian on an open neighborhood of 0.

With the same argument, it is not difficult to show that, if $\delta$, $(\varepsilon_1,\varepsilon_2)$ and $\| y \|$ are small enough, then the mapping $u \mapsto \Phi(\varepsilon_1,\varepsilon_2,y,u)$ maps $U_{\delta}$ into itself. Lemma \ref{iFTP} follows from the application of the usual Banach fixed point theorem to the contraction mapping $u \mapsto \Phi(\varepsilon_1,\varepsilon_2,y,u)$ with parameters $(\varepsilon_1,\varepsilon_2,y)$.

\newpage

\bibliographystyle{unsrt}
\bibliography{references}

\begin{thebibliography}{10}

\bibitem{Malek-Zavarei:1987:TSA:576371}
Manu Malek-Zavarei and Mohammad Jamshidi.
\newblock {\em Time-Delay Systems: Analysis, Optimization and Applications}.
\newblock Elsevier Science Inc., New York, NY, USA, 1987.

\bibitem{wong1985optimal}
KH~Wong, DJ~Clements, and KL~Teo.
\newblock Optimal control computation for nonlinear time-lag systems.
\newblock {\em Journal of Optimization Theory and Applications}, 47(1):91--107,
  1985.

\bibitem{lee1993numerical}
Llan~Y Lee.
\newblock Numerical solution of time-delayed optimal control problems with
  terminal inequality constraints.
\newblock {\em Optimal Control Applications and Methods}, 14(3):203--210, 1993.

\bibitem{chen2000numerical}
Cheng-Liang Chen, Daim-Yuang Sun, and Chia-Yuan Chang.
\newblock Numerical solution of time-delayed optimal control problems by
  iterative dynamic programming.
\newblock {\em Optimal control applications and methods}, 21(3):91--105, 2000.

\bibitem{lee2006semi}
HWJ Lee and KH~Wong.
\newblock Semi-infinite programming approach to nonlinear time-delayed optimal
  control problems with linear continuous constraints.
\newblock {\em Optimisation Methods and Software}, 21(5):679--691, 2006.

\bibitem{gollmann2009optimal}
L~Gollmann, D~Kern, and H~Maurer.
\newblock Optimal control problems with delays in state and control variables
  subject to mixed control-state constraints.
\newblock {\em Optimal Control Applications and Methods}, 30(4):341--365, 2009.

\bibitem{horng1985analysis}
Ing-Rong Horng and Jyh-Horng Chou.
\newblock Analysis, parameter estimation and optimal control of time-delay
  systems via chebyshev series.
\newblock {\em International Journal of Control}, 41(5):1221--1234, 1985.

\bibitem{hwang1985optimal}
C~Hwang and YP~Shih.
\newblock Optimal control of delay systems via block pulse functions.
\newblock {\em Journal of Optimization Theory and Applications},
  45(1):101--112, 1985.

\bibitem{perng1986direct}
MING-HWEI PERNG.
\newblock Direct approach for the optimal control of linear time-delay systems
  via shifted legendre polynomials.
\newblock {\em International Journal of Control}, 43(6):1897--1904, 1986.

\bibitem{khellat2009optimal}
F~Khellat.
\newblock Optimal control of linear time-delayed systems by linear legendre
  multiwavelets.
\newblock {\em Journal of optimization theory and applications},
  143(1):107--121, 2009.

\bibitem{haddadi2012optimal}
N~Haddadi, Yadollah Ordokhani, and Mohsen Razzaghi.
\newblock Optimal control of delay systems by using a hybrid functions
  approximation.
\newblock {\em Journal of Optimization Theory and Applications},
  153(2):338--356, 2012.

\bibitem{banks1979approximation}
HT~Banks.
\newblock Approximation of nonlinear functional differential equation control
  systems.
\newblock {\em Journal of Optimization Theory and Applications},
  29(3):383--408, 1979.

\bibitem{trelat2012optimal}
Emmanuel Tr{\'e}lat.
\newblock Optimal control and applications to aerospace: some results and
  challenges.
\newblock {\em Journal of Optimization Theory and Applications},
  154(3):713--758, 2012.

\bibitem{pontryagin1987mathematical}
Lev~Semenovich Pontryagin.
\newblock {\em Mathematical theory of optimal processes}.
\newblock CRC Press, 1987.

\bibitem{kharatishvili1961maximum}
GL~Kharatishvili.
\newblock The maximum principle in the theory of optimal processes involving
  delay.
\newblock {\em Doklady Akademii Nauk SSSR}, 136(1):39--+, 1961.

\bibitem{guinn1976reduction}
T~Guinn.
\newblock Reduction of delayed optimal control problems to nondelayed problems.
\newblock {\em Journal of Optimization Theory and Applications},
  18(3):371--377, 1976.

\bibitem{kharatishvili1967maximum}
GL~Kharatishvili.
\newblock A maximum principle in extremal problems with delays.
\newblock {\em Mathematical Theory of Control}, pages 26--34, 1967.

\bibitem{halanay1968optimal}
Andrei Halanay.
\newblock Optimal controls for systems with time lag.
\newblock {\em SIAM Journal on Control}, 6(2):215--234, 1968.

\bibitem{soliman1972optimal}
MA~Soliman and WH~Ray.
\newblock On the optimal control of systems having pure time delays and
  singular arcs 1, some necessary conditions for optimality.
\newblock {\em International Journal of Control}, 16(5):963--976, 1972.

\bibitem{banks1968necessary}
HT~Banks.
\newblock Necessary conditions for control problems with variable time lags.
\newblock {\em SIAM Journal on Control}, 6(1):9--47, 1968.

\bibitem{gollmann2014theory}
Laurenz G{\"o}llmann and Helmut Maurer.
\newblock Theory and applications of optimal control problems with multiple
  time-delays.
\newblock {\em Journal of Industrial and Management Optimization},
  10(2):413--441, 2014.

\bibitem{boccia2013free}
Andrea Boccia, Paola Falugi, Helmut Maurer, and Richard~B Vinter.
\newblock Free time optimal control problems with time delays.
\newblock In {\em Decision and Control (CDC), 2013 IEEE 52nd Annual Conference
  on}, pages 520--525. IEEE, 2013.

\bibitem{mallet2003mixed}
JOHN Mallet-Paret and SM~Verduyn Lunel.
\newblock Mixed-type functional differential equations, holomorphic
  factorization and applications.
\newblock In {\em Proceedings Equadiff}, pages 73--89, 2003.

\bibitem{ford2010numerical}
Neville~J Ford, Patricia~M Lumb, Pedro~M Lima, and M~Filomena Teodoro.
\newblock The numerical solution of forward--backward differential equations:
  Decomposition and related issues.
\newblock {\em Journal of computational and applied mathematics},
  234(9):2745--2756, 2010.

\bibitem{bader1985solving}
G~Bader.
\newblock Solving boundary value problems for functional differential equations
  by collation.
\newblock In {\em Numerical Boundary Value ODEs}, pages 227--243. Springer,
  1985.

\bibitem{allgower2003introduction}
Eugene~L Allgower and Kurt Georg.
\newblock {\em Introduction to numerical continuation methods}, volume~45.
\newblock SIAM, 2003.

\bibitem{filippov1962certain}
A.F Filippov.
\newblock On certain questions in the theory of optimal control.
\newblock {\em Journal of the Society for Industrial and Applied Mathematics,
  Series A: Control}, 1(1):76--84, 1962.

\bibitem{nababan1979filippov}
S~Nababan.
\newblock A filippov-type lemma for functions involving delays and its
  application to time-delayed optimal control problems.
\newblock {\em Journal of Optimization Theory and Applications},
  27(3):357--376, 1979.

\bibitem{clarke1987relationship}
Frank~H Clarke and Richard~B Vinter.
\newblock The relationship between the maximum principle and dynamic
  programming.
\newblock {\em SIAM Journal on Control and Optimization}, 25(5):1291--1311,
  1987.

\bibitem{aubin2009set}
Jean-Pierre Aubin and H{\'e}l{\`e}ne Frankowska.
\newblock {\em Set-valued analysis}.
\newblock Springer Science \& Business Media, 2009.

\bibitem{emmRifford}
L.~Rifford and E~Tr\'{e}lat.
\newblock On the stabilization problem for nonholonomic distributions.
\newblock {\em J. Eur. Math. Soc.}, 11(2):223--255, 2009.

\bibitem{visintin1984strong}
Augusto Visintin.
\newblock Strong convergence results related to strict convexity.
\newblock {\em Communications in Partial Differential Equations},
  9(5):439--466, 1984.

\bibitem{asher1971optimal}
Robert~B Asher and Henry~R Sebesta.
\newblock Optimal control of systems with state-dependent time delay.
\newblock {\em International Journal of Control}, 14(2):353--366, 1971.

\bibitem{dadebo1992optimal}
Solomon Dadebo and Rein Luus.
\newblock Optimal control of time-delay systems by dynamic programming.
\newblock {\em Optimal Control Applications and Methods}, 13(1):29--41, 1992.

\bibitem{oh1976optimal}
Se~H Oh and Rein Luus.
\newblock Optimal feedback control of time-delay systems.
\newblock {\em AIChE Journal}, 22(1):140--147, 1976.

\bibitem{hybrd}
J.~J. Mor\'{e}, D.~C. Sorensen, K.~E. Hillstrom, and B.~S. Garbow.
\newblock The minpack project.
\newblock {\em Sources and Development of Mathematical Software}, pages
  88–--111, 1984.

\bibitem{AMPL}
Robert Fourer, David Gay, and Brian Kernighan.
\newblock {\em Ampl}, volume 117.
\newblock Boyd \& Fraser Danvers, MA, 1993.

\bibitem{wachter2006implementation}
Andreas W{\"a}chter and Lorenz~T Biegler.
\newblock On the implementation of an interior-point filter line-search
  algorithm for large-scale nonlinear programming.
\newblock {\em Mathematical programming}, 106(1):25--57, 2006.

\bibitem{balas2012adaptive}
Mark~J Balas.
\newblock Adaptive control of nonminimum phase systems using sensor blending
  with application to launch vehicle control.
\newblock In {\em Conference on Smart Materials, Adaptive Structures and
  Intelligent Systems. Stone Mountain}, 2012.

\bibitem{myIFAC2017}
Riccardo Bonalli, Bruno H\'{e}riss\'{e}, and Emmanuel Tr\'{e}lat.
\newblock Analytical initialization of a continuation-based indirect method for
  optimal control of endo-atmospheric launch vehicle systems.
\newblock In {\em IFAC World Congress}, 2017.

\bibitem{silva2010smooth}
Cristiana Silva and Emmanuel Tr{\'e}lat.
\newblock Smooth regularization of bang-bang optimal control problems.
\newblock {\em IEEE Transactions on Automatic Control}, 55(11):2488--2499,
  2010.

\bibitem{haberkorn2011convergence}
Thomas Haberkorn and Emmanuel Tr{\'e}lat.
\newblock Convergence results for smooth regularizations of hybrid nonlinear
  optimal control problems.
\newblock {\em SIAM Journal on Control and Optimization}, 49(4):1498--1522,
  2011.

\bibitem{lee1967foundations}
Ernest~Bruce Lee and Lawrence Markus.
\newblock Foundations of optimal control theory.
\newblock Technical report, DTIC Document, 1967.

\bibitem{gamkrelidze2013principles}
R~Gamkrelidze.
\newblock {\em Principles of optimal control theory}, volume~7.
\newblock Springer Science \& Business Media, 2013.

\bibitem{agrachev2013control}
Andrei~A Agrachev and Yuri Sachkov.
\newblock {\em Control theory from the geometric viewpoint}, volume~87.
\newblock Springer Science \& Business Media, 2013.

\bibitem{myACC2017}
Riccardo Bonalli, Bruno H{\'e}riss{\'e}, and Emmanuel Tr{\'e}lat.
\newblock Solving optimal control problems for delayed control-affine systems
  with quadratic cost by numerical continuation.
\newblock In {\em American Control Conference}, 2017.

\bibitem{trelat2008controle}
Emmanuel Tr{\'e}lat.
\newblock {\em Contr{\^o}le optimal: th{\'e}orie \& applications}.
\newblock Vuibert, 2008.

\end{thebibliography}
\end{document}